\documentclass[a4paper,10pt]{article}

\def\zibreport{1}
\def\longtitle{Solving Quadratic Programs to High Precision using Scaled Iterative Refinement}
\def\shortfunding{This project has received funding from the European Research Council (ERC) under the European Union’s Horizon 2020 research and innovation programme (grant agreement No 647573), from the Deutsche Forschungsgemeinschaft (DFG, German Research Foundation) -- 314838170, GRK 2297 MathCoRe, and from the German Federal Ministry of Education and Research as part of the \emph{Research Campus MODAL} (BMBF grant number 05M14ZAM), all of which is gratefully acknowledged.}
\def\shortauthors{Tobias Weber, Sebastian Sager, Ambros Gleixner}

\usepackage{geometry}
\usepackage[colorlinks,citecolor=blue,linkcolor=blue,urlcolor=black,
  pdftitle={\longtitle},
  pdfauthor={\shortauthors},
  pdfsubject={\longtitle}]{hyperref}

\usepackage{amsmath,amsthm}
\theoremstyle{plain}
\newtheorem{theorem}{Theorem}

\usepackage{titling}
\pretitle{\begin{center}\linespread{1.05}\Large}
\posttitle{\par\end{center}\vskip 0.5em}
\preauthor{\begin{center}\linespread{1.15}\large \lineskip 0.6em\begin{tabular}[t]{c}}
\postauthor{\end{tabular}\par\end{center}\vskip 0.6em}
\predate{\begin{center}}
\postdate{\par\end{center}}
\thanksmarkseries{arabic}

\usepackage[]{titlesec}
\usepackage{etoolbox}
\makeatletter
\patchcmd{\ttlh@hang}{\parindent\z@}{\parindent\z@\leavevmode}{}{}
\patchcmd{\ttlh@hang}{\noindent}{}{}{}
\makeatother
\titleformat{\section}
{\normalfont\large\bfseries}{\thesection}{0.9ex}{}
\titleformat{\subsection}
{\normalfont\normalsize\bfseries}{\thesubsection}{0.9ex}{}
\titleformat{\subsubsection}
{\normalfont\normalsize\upshape}{\thesubsubsection}{0.9ex}{}
\titleformat{\paragraph}[runin]
{\normalfont\normalsize\itshape}{\theparagraph}{1em}{}
\titleformat{\subparagraph}[runin]
{\normalfont\normalsize\itshape}{\theparagraph}{1em}{}
\titlespacing*{\section}     {0pt}{21dd plus 8pt minus 4pt}{10.5dd}
\titlespacing*{\subsection}   {0pt}{21dd plus 8pt minus 4pt}{10.5dd}
\titlespacing*{\subsubsection}{0pt}{19dd plus 8pt minus 4pt}{10.5dd}
\titlespacing*{\paragraph}   {0pt}{13pt plus 8pt minus 4pt}{1em}
\titlespacing*{\subparagraph}   {0pt}{13pt plus 8pt minus 4pt}{1em}

\usepackage[numbers]{natbib}

\usepackage{graphicx}
\usepackage{algorithm}
\usepackage[noend]{algpseudocode}
\usepackage{longtable}
\usepackage{color}
\usepackage{pdflscape}
\usepackage{bbm}
\usepackage{amssymb}
\usepackage{xspace}
\usepackage{pgfplots}
\usepackage{booktabs}
\usepackage{multirow}
\usepackage{colortbl}
\usepackage{booktabs}		\pgfplotsset{compat=1.5.1}
\usepackage[textsize=small]{todonotes}

\newcommand{\mline}[1]{\textit{Line~#1}\xspace}
\newcommand{\mlines}[1]{\textit{Lines~#1}\xspace}

\newtheorem{assumption}[theorem]{Assumption}
\newtheorem{mylemma}[theorem]{Lemma}
\newtheorem{mydefinition}[theorem]{Definition}
\newtheorem{myexample}[theorem]{Example}
\newtheorem{myremark}[theorem]{Remark}
 
\ifthenelse{\zibreport = 1}{
  \let\pdfoutorg\pdfoutput
  \let\pdfoutput\undefined
  \usepackage{zibtitlepage}
  \let\pdfoutput\pdfoutorg

  \ZTPTitle{\longtitle}
  \ZTPAuthor{\shortauthors}
  \ZTPPreprint
  \ZTPNumber{18-04}
  \ZTPMonth{March}
  \ZTPYear{2018}
  \ZTPInfo{This report has been published in \textit{Mathematical Programming Computation}. Please cite as \textit{Tobias Weber, Sebastian Sager, Ambros Gleixner, Solving Quadratic Programs to High Precision using Scaled Iterative Refinement, Mathematical Programming Computation, 11:421--455, 2019, \href{https://doi.org/10.1007/s12532-019-00154-6}{doi:10.1007/s12532-019-00154-6}}.}
}{}

\begin{document}

\title{Solving Quadratic Programs to High Precision\\ using Scaled Iterative Refinement}

\author{Tobias Weber  \thanks{\itshape Otto von Guericke University, Institute of Mathematical Optimization,
    Universit\"atsplatz 2, 02-204, 39106 Magdeburg, Germany,
    \nolinkurl{{tobias.weber,sager}@ovgu.de}\vspace*{0.5ex}},
  Sebastian Sager\thanksmark{1},
  Ambros Gleixner  \thanks{\itshape Zuse Institute Berlin, Department of Mathematical Optimization,
    Takustr.~7, 14195~Berlin, Germany,
    \nolinkurl{gleixner@zib.de}\vspace*{0.5ex}
    \footnoterule\upshape\noindent
    \shortfunding
  }}

\date{March 19, 2018}

\ifthenelse{\zibreport = 1}{\zibtitlepage}{}

\newgeometry{left=38mm,right=38mm,top=35mm}

\maketitle

\paragraph{\bf Abstract}

Quadratic optimization problems (QPs) are ubiquitous, and solution algorithms have matured to a reliable technology. However, the precision of solutions is usually limited due to the underlying floating-point operations. This may cause inconveniences when solutions are used for rigorous reasoning. We contribute on three levels to overcome this issue.

First, we present a novel refinement algorithm to solve QPs to arbitrary precision. It iteratively solves refined QPs, assuming a floating-point QP solver oracle. We prove linear convergence of residuals and primal errors.
Second, we provide an efficient implementation, based on \texttt{SoPlex} and \texttt{qpOASES} that is publicly available in source code.
Third, we give precise reference solutions for the Maros and M{\'e}sz{\'a}ros benchmark library.



\section{Introduction} \label{sec:intro}

Quadratic optimization problems (QPs) are optimization problems with a quadratic objective function and linear constraints. They are of interest directly, e.g., in portfolio optimization or support vector machines~\cite{bennett2000}. They also occur as subproblems in sequential quadratic programming, mixed-integer quadratic programming, and nonlinear model predictive control. Efficient algorithms are usually of active set, interior point, or parametric type. Examples of QP solvers are \texttt{BQPD}~\cite{fletcher1998user}, \texttt{CPLEX}~\cite{cplex2016}, \texttt{Gurobi}~\cite{gurobimanual2017}, \texttt{qp\_solve}~\cite{goldfarb1982dual}, \texttt{qpOASES}~\cite{ferreau2014qpoases}, and \texttt{QPOPT}~\cite{gill1995user}.
These QP solvers have matured to reliable tools and can solve convex problems with many thousands, sometimes millions of variables. However, they calculate and check the solution of a QP in floating-point arithmetic. Thus, the claimed precision may be violated and in extreme cases optimal solutions might not be found.
This may cause inconveniences, especially when solutions are used for rigorous reasoning.

One possible approach is the application of interval arithmetic. It allows to include uncertainties as lower and upper bounds on the modeling level, see~\cite{Hladik2012} for a survey for the case of linear optimization. As a drawback, all internal calculations have to be performed with interval arithmetic. Hence standard solvers can not be used any more, the computation times increase, and solutions may be very conservative.

We are only aware of one advanced algorithm that solves QPs exactly over the rational numbers. It is designed to tackle problems from computational geometry with a small number of constraints or variables, \cite{gartner2000efficient}.
Based on the classical QP simplex method~\cite{wolfe1959}, it replaces critical calculations inside the QP solver by their rational counterparts.  Heuristic decisions that do not affect the correctness of the algorithm are performed in fast floating-point arithmetic.

In this paper we propose a novel algorithm that can use efficient floating-point QP solvers as a black box. Our method is inspired by iterative refinement, a standard procedure to improve the accuracy of an approximate solution for a system of linear equalities, \cite{Wilkinson1959}: The residual of the approximate solution is calculated, the linear system is solved again with the residual as a right-hand side, and the new solution is used to refine the old solution, thus improving its accuracy. 
A generalization of this idea to the solution of optimization problems needs to address several difficulties: most importantly, the presence of inequality constraints; the handling of optimality conditions; and the numerical tolerances that floating-point solvers can return in practice.

For LPs this has first been developed in~\cite{gleixner2012}.
The approach refines  primal-dual solutions of the Karush-Kuhn-Tucker (KKT) conditions and comprehends scaling and calculations in rational arithmetic.
We generalize further and discuss the specific issues due to the presence of a quadratic objective function.
The fact that the general approach carries over from LP to QP was remarked in~\cite{gleixner2015}. Here we provide the details, provide a general lemma showing how the residuals bound the primal and dual iterates, and analyze the computational behavior of the algorithm based on an efficient implementation that is made publicly available in source code and can be used freely for research purposes.

The paper is organized as follows. In Section~\ref{sec:IRQP} we define and discuss QPs and their refined and scaled counterparts. We give one illustrating and motivating example for scaling and refinement. In Section~\ref{sec:algorithm} we formulate an algorithm and prove its convergence properties. In Section~\ref{sec:implementation} we consider performance issues and describe how our implementation based on \texttt{SoPlex} and \texttt{qpOASES} can be used to calculate solutions for QPs with arbitrary precision. In Section~\ref{sec:NumericalExperiments} we discuss run times and provide solutions for the Maros and M{\'e}sz{\'a}ros benchmark library, \cite{maros1999repository}. We conclude in Section~\ref{sec:Conclusion} with a discussion of the results and give directions for future research and applications of the algorithm.

In the following we will use $\|\cdot\|$ for the maximum norm $\|\cdot\|_\infty$. The maximal entry of a vector $\max_i\{v_i\}$ is written as $\max\{v\}$. Inequalities $a\leq b$ for $a,b\in \mathbb{Q}^n$ hold componentwise.

\section{Refinement and Scaling of Quadratic Programs}
\label{sec:IRQP}

In this section we collect some basic definitions and results that will be of use later on.
We consider convex optimization problems of the following form.
\begin{mydefinition}[Quadratic optimization problem (QP)]
Let  a symmetric matrix $Q \in \mathbb{Q}^{n\times n}$, a matrix $A \in \mathbb{Q}^{m\times n}$, and vectors $c \in \mathbb{Q}^n, b \in \mathbb{Q}^m, l \in \mathbb{Q}^n$ be given. We consider the quadratic optimization problem (QP) 
\begin{equation}\label{P}\tag{$P$}
\begin{array}{rrclcl}
\displaystyle \min_{x} & \multicolumn{3}{l}{\frac{1}{2} x^T Q x + c^T x} \\
\textrm{s.t.} & A x & = & b \\
& x & \geq & l & &  \\
\end{array}
\end{equation}
assuming that (\ref{P}) is feasible and bounded, and $Q$ is positive semi-definite on the feasible set.
The rounded version of this convex rational data QP will be denoted as $(\tilde{P})$.
\end{mydefinition}
A point $x^* \in \mathbb{Q}^n$ is a global optimum of (\ref{P}) if and only if it satisfies the Karush-Kuhn-Tucker (KKT) conditions, \cite{Dostal2009}, i.e., if multipliers $y^* \in \mathbb{Q}^m$ exist such that
\begin{subequations} \label{eq:kkt}
\begin{eqnarray}
 A x^* & = & b \label{eq:kkta} \\
 x^*  & \geq & l \label{eq:kktb} \\
 A^Ty^* & \leq & Qx^* + c \label{eq:kktc} \\
 (Qx^* + c - A^Ty^*)^T(x^*-l) & = & 0. \label{eq:kktd} 
\end{eqnarray}
\end{subequations}
The pair $(x^*,y^*)$ is then called KKT pair of (\ref{P}). Primal feasibility is given by (\ref{eq:kkta}--\ref{eq:kktb}), dual feasibility by \eqref{eq:kktc}, and complementary slackness by \eqref{eq:kktd}. Refinement of this system of linear (in-)equalities is equivalent to the refinement of \eqref{P}.

\begin{mydefinition}[Refined QP]
Let the QP (\ref{P}), scaling factors $\Delta_P,\Delta_D \in \mathbb{Q}_+$ and vectors $x^* \in \mathbb{Q}^n,y^* \in \mathbb{Q}^m$ be given. We define the refined QP as
\begin{equation}\label{PP}\tag{$P^{\Delta}$}
\begin{array}{rrclcl}
\displaystyle \min_{x} & \multicolumn{3}{l}{\frac{1}{2} x^T \frac{\Delta_D}{\Delta_P} Q x + (\Delta_D\hat{c})^T x} \\
\textrm{s.t.} & A x & = & \Delta_P\hat{b} \\
& x & \geq & \Delta_P\hat{l}, & &  \\
\end{array}
\end{equation}
where $\hat{c}=Qx^*+c-A^Ty^*$, $\hat{b}=b-Ax^*$, and $\hat{l}=l-x^*$.
The rounded version of this refined and scaled rational data QP will be denoted as ($\tilde{P}^{\Delta}$).
\end{mydefinition}
The following theorem is the basis for our theoretical and algorithmic approaches.
It is a generalization of iterative refinement for LP and was formulated and proven in~\cite[Theorem 5.2]{gleixner2015}.
Again, primal feasibility refers to (\ref{eq:kkta}--\ref{eq:kktb}), dual feasibility to \eqref{eq:kktc}, and complementary slackness to \eqref{eq:kktd}.
\begin{theorem}[QP Refinement]\label{th:refinement}
Let the QP (\ref{P}), scaling factors $\Delta_P,\Delta_D \in \mathbb{Q}_+$, vectors $x^* \in \mathbb{Q}^n,y^* \in \mathbb{Q}^m$, and the refined QP \eqref{PP} be given.
Then for any $\hat{x} \in \mathbb{R}^n$, $\hat{y} \in \mathbb{R}^m$ and tolerances $\epsilon_P,\epsilon_D,\epsilon_S \geq 0$:
\begin{enumerate}
\item $\hat{x}$ is primal feasible for $($\ref{PP}$)$ within an absolute tolerance $\epsilon_P$ if and only if $x^* + \frac{\hat{x}}{\Delta_P}$ is primal feasible for $($\ref{P}$)$ within $\epsilon_P/\Delta_P$.
\item $\hat{y}$ is dual feasible for $($\ref{PP}$)$ within an absolute tolerance $\epsilon_D$ if and only if $y^* + \frac{\hat{y}}{\Delta_D}$ is dual feasible for $($\ref{P}$)$ within $\epsilon_D/\Delta_D$.
\item $\hat{x}$, $\hat{y}$ satisfy complementary slackness for $($\ref{PP}$)$ within an absolute tolerance $\epsilon_S$ if and only if $y^* + \frac{\hat{y}}{\Delta_D}$, $x^* + \frac{\hat{x}}{\Delta_P}$ satisfy complementary slackness for $($\ref{P}$)$ within $\epsilon_S/(\Delta_P\Delta_D)$.
\end{enumerate}
\end{theorem}

For illustration, we investigate the following example.
\begin{myexample}[QP Refinement]\label{example}
\normalfont
Consider the QP with two variables
\begin{equation*}
\begin{array}{rrclcl}
\displaystyle \min_{x} & \multicolumn{3}{l}{\frac{1}{2} (x_1^2 + x_2^2) + x_1 + (1+10^{-6})x_2} \\
\textrm{s.t.} & x_1 + x_2 & = & 10^{-6} \\
& x_1,x_2 & \geq & 0. & &  \\
\end{array}
\end{equation*}
An approximate solution to a tolerance of $10^{-6}$ is $x^*_1=x^*_2=0$ with dual multiplier $y^*=1$. This solution is slightly primal and dual infeasible, but the solver can not recognize this on this scale. The situation is depicted in Figure~\ref{fig:qpex1} on the left.

\begin{figure}[h!!!]
\caption{Illustration of nominal QP (left) and of refined QP (right) for Example~\ref{example}. The scaled (and shifted) QP \eqref{PP} works like a zoom (and shift) for \eqref{P}, allowing to correct the solution $x^*$ (orange dot) from $(0,0)$ to $(10^{-6},0)$. \label{fig:qpex1}}
\vspace{0.5cm}
\begin{tikzpicture}
\begin{axis}[ 
    axis lines = middle,
    axis line style={->},
    ymin=-0.9, ymax=1.1,
    xmin=-1.1, xmax=0.9,
    xlabel={$x_1$},
    ylabel={$x_2$},
    axis equal image
]
\addplot[blue, domain=-1:1] {-x}
   node [pos=0.1, above right] {$x_1+x_2=10^{-6}$};
\addplot[orange, mark=*, only marks] coordinates {(0,0)};
\node[label={20:$x^*$}] (orange) at (axis cs:0,0) {};
\draw[brown] (axis cs:-1,-1) circle [radius=1.41];
\node[brown] (levelcurve) at (axis cs:-0.5,0.1) {level curve};
\end{axis}
\end{tikzpicture}
\hfill
\begin{tikzpicture}
\begin{axis}[ 
    axis lines = middle,
    axis line style={->},
    ymin=-0.9, ymax=1.1,
    xmin=-0.1, xmax=1.6,
    xlabel={$x_1$},
    ylabel={$x_2$},
    axis equal image
]
\addplot[blue, domain=0.0:1.5] {-x+1}
   node [pos=0.3, above right] {$x_1+x_2=1$};
\addplot[orange, mark=*, only marks] coordinates {(0,0)};
\addplot[purple, mark=*, only marks] coordinates {(1,0)};
\node[label={20:$x^*$}] (orange) at (axis cs:0,0) {};
\node[label={20:$\hat{x}$}] (purple) at (axis cs:1,0) {};
\draw[dashed,-latex] (orange) to [bend right] (purple);
\draw[brown] (axis cs:0,-1) circle [radius=1.41];
\node[brown] (levelcurve) at (axis cs:1,-0.6) {level curve};
\end{axis}
\end{tikzpicture}
\end{figure}
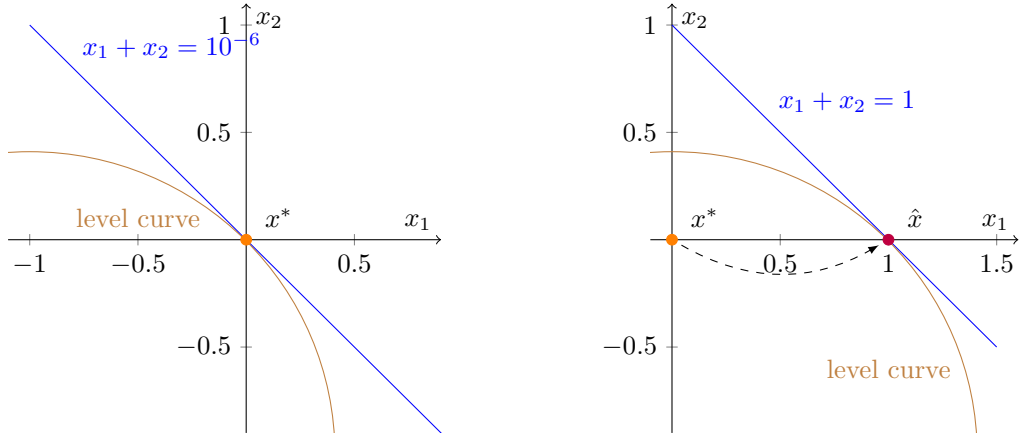

The point $x^*$ seems to be the optimal solution satisfying the equality constraint and the brown circle representing the level curve of the objective function indicates the optimality.
The corresponding violations are $\hat{l}=(0,0)^T$, $\hat{b}=10^{-6}$, and $\hat{c}=(0,10^{-6})^T$. The refined QP is
\begin{equation*}
\begin{array}{rrclcl}
\displaystyle \min_{x} & \multicolumn{3}{l}{\frac{1}{2} (x_1^2 + x_2^2) + x_2} \\
\textrm{s.t.} & x_1 + x_2 & = & 1 \\
& x_1,x_2 & \geq & 0 & &  \\
\end{array}
\end{equation*}
with scalings $\Delta_P=\Delta_D=10^6$. The optimal solution to this problem is $\hat{x}_1=1,\hat{x}_2=0$ with multiplier $\hat{y}=1$. This situation is depicted in Figure~\ref{fig:qpex1}, right. The point $x^*$ is obviously not the optimal solution and the solution to the refined problem is $\hat{x}$. The refined solution is $x^*+\hat{x}/\Delta_P=(10^{-6},0)^T$ and $y^*+\hat{y}/\Delta_D=1+10^{-6}$. These values are primal and dual feasible in the original problem.
\end{myexample}

\section{Iterative Refinement for Quadratic Programming}
\label{sec:algorithm}

To solve quadratic programs to arbitrary, \textit{a priori} specified precisions,
we apply the refinement idea from the previous section iteratively in Algorithm~\ref{IQPR}.
\begin{algorithm}[h!!!]
\caption{Iterative QP Refinement (IQPR)}\label{IQPR}
\begin{algorithmic}[1]
\item \textbf{Input:} (\ref{P}) in rational precision, termination tolerances $\epsilon_P$ and $\epsilon_D$ and $\epsilon_S$, scaling limit $\alpha > 1$, iteration limit $k_{max}$\item \textbf{Initialization:} $\Delta_0 \gets 1$, solve ($\tilde{P}$) approximately, save optimal $(x_1,y_1)$\For{$k \gets 1,2,...,k_{max}$}
\State $\hat{b} \gets b-Ax_k$
\State $\hat{l} \gets l-x_k$
\State $\delta_P \gets \max\left\{0,\|\hat{b}\|,\max\{\hat{l}\}\right\}$
\State $\hat{c} \gets Qx_k+c-A^Ty_k$
\State $\delta_D \gets \max\left\{0,\max\{-\hat{c}\}\right\}$
\State $\Delta_k \gets \min\left\{\delta_P^{-1},\delta_D^{-1},\alpha \Delta_{k-1}\right\}$
\State $\delta_S \gets \sum_i \hat{l}_i\hat{c}_i$
\If{$\delta_P \leq \epsilon_P$ and $\delta_D \leq \epsilon_D$ and $\delta_S \leq \epsilon_S$}
\State break
\Else
\State solve ($\tilde{P}^{\Delta_k}$) approximately
\State $(x^*,y^*) \gets$ KKT pair returned as optimal
\State $(x_{k+1},y_{k+1}) \gets (x_k,y_k) + \frac{(x^*,y^*)}{\Delta_k}$
\EndIf
\EndFor
\item \textbf{Return:} $(x_k,y_k)$\end{algorithmic}
\end{algorithm}

Algorithm~\ref{IQPR} expects QP data ($Q,A,c,b,l$) in rational precision, primal and dual termination tolerances~($\epsilon_P,\epsilon_D$), complementary slack termination tolerance~($\epsilon_S$), scaling limit $\alpha > 1$ and iteration limit $k_{max}$. First the rounded QP ($\tilde{P}$) is solved with a floating-point QP solver oracle which returns optimal primal and dual solution vectors (\mline{2}).
In \mline{3} the main loop begins. The primal violations for constraints ($\hat{b}$, \mline{4}) and for bounds ($\hat{l}$, \mline{5}) are calculated. The maximal primal violation is saved as $\delta_P$ in \mline{6}. The reduced cost vector~$\hat{c}$ and its maximal violation $\delta_D$ are calculated in \mlines{7--8}. In \mline{9} the scaling factor $\Delta_k$ is chosen as the maximum of $\alpha\Delta_{k-1}$ and the inverses of the violations $\delta_P$ and $\delta_D$. 
The complementary slack violation~$\delta_S$ is calculated in \mline{10}.
If the primal, dual and complementary slack violations are already below the specified tolerances the loop is stopped (\mlines{11--12}) and the optimal solution is returned (\mline{17}).
Else (\mline{13}) the refined, scaled, and rounded QP ($\tilde{P}^{\Delta_k}$) is solved with the floating-point QP oracle in \mline{14}. We save the floating-point optimal primal and dual solution vectors (\mline{15}). We scale and add them to the current iterate ($x_k,y_k$) to obtain ($x_{k+1},y_{k+1}$), \mline{16}. 

Note that all calculations except the expensive solves of the QPs are done in rational precision. Algorithm~\ref{IQPR} uses only one scaling factor $\Delta_k$ for primal and dual infeasibility to avoid the scaling of the quadratic term of the objective. Keeping this matrix and the constraint matrix~$A$ fixed gives QP solvers the possibility to reuse the internal factorization of the basis system between refinements, as the transformation does not change the basis. This allows to use efficient hotstart techniques for all solves after the initial solve.

To investigate the performance of the algorithm we make, in analogy with the LP case~\cite[Assumption 2.9]{gleixner2012}, the following assumption.
\begin{assumption}[QP solver accuracy]\label{assumption}
We assume that there exists $\epsilon \in [0,1)$ and $\sigma \geq 0$ such that the QP solver oracle returns for all rounded QPs ($\tilde{P}^{\Delta_k}$) solutions $(\bar{x},\bar{y})$ that satisfy
\begin{equation*}
\begin{array}{rcl}
 \|A \bar{x} - \Delta_k\hat{b}\| & \leq & \epsilon \\
 \bar{x}  & \geq & \Delta_k\hat{l} - \mathbbm{1} \epsilon \\
 Q\bar{x} + \Delta_k\hat{c} & \geq & A^T\bar{y} - \mathbbm{1} \epsilon \\
 |(Q\bar{x} + \Delta_k\hat{c} - A^T\bar{y})^T(\bar{x}-\Delta_k\hat{l})| & \leq & \sigma \\
\end{array}
\end{equation*}
in the rational QPs ($P^{\Delta_k}$).
\end{assumption}
Note that this $\epsilon$ corresponds to a termination tolerance passed to a floating-point solver, while the algorithm uses overall termination tolerances $\epsilon_P$ and $\epsilon_D$ and a scaling limit $\alpha > 1$ per iteration. We denote $\tilde{\epsilon}=\max\{1/\alpha,\epsilon\}$.

\begin{mylemma}[Termination and residual convergence] \label{le:rightside}
Algorithm~\ref{IQPR} applied to a primal and dual feasible QP (\ref{P}) and using a QP solver that satisfies Assumption~\ref{assumption} will terminate in at most
\begin{equation}\label{eq:iterationnumber}
 k_{max} = \max\left\{\;\log(\epsilon_P)/\log(\tilde{\epsilon}),\;\log(\epsilon_D)/\log(\tilde{\epsilon}),\;\log(\epsilon_S/\sigma)/(2\log(\tilde{\epsilon}))+1\;\right\}
\end{equation}
iterations. Furthermore, after each iteration $k=1,2,...$ the primal-dual iterate ($x_k,y_k$) and the scaling factor $\Delta_k$ satisfy
\begin{subequations}\label{eq:iterate}
\begin{eqnarray}
 \Delta_k & \geq & 1/\tilde{\epsilon}^{k} \label{eq:iteratea}\\
 \|A x_k - b\| & \leq & \tilde{\epsilon}^{k} \label{eq:iterateb}\\
 x_k - l  & \geq & - \mathbbm{1} \tilde{\epsilon}^{k} \label{eq:iteratec}\\
 Qx_k + c - A^Ty_k & \geq &  - \mathbbm{1} \tilde{\epsilon}^{k} \label{eq:iterated}\\
 |(Qx_k + c - A^Ty_k)^T(x_k-l)| & \leq & \sigma\tilde{\epsilon}^{2(k-1)}. \label{eq:iteratee}
\end{eqnarray}
\end{subequations}
\end{mylemma}
\begin{proof}
We prove (\ref{eq:iterate}) by induction over $k$, starting with $k=1$.
As $\tilde{\epsilon} \ge \epsilon$, the claims (\ref{eq:iterateb}--\ref{eq:iteratee}) follow directly from Assumption~\ref{assumption}.
Using \mlines{6, 4--5}, and Assumption~\ref{assumption} we obtain
$$\delta_P = \max\left\{0,\|\hat{b}\|,\max\{\hat{l}\}\right\} = \max\{0,\| Ax_1 - b \|, \max\{l - x_1\} \} \leq \epsilon$$
and with \mlines{8,7} and Assumption~\ref{assumption}
$$\delta_D = \max\left\{0,\max\{ -\hat{c}\}\right\} = \max\left\{0,\max\{Qx_1+c-A^Ty_1\}\right\} \leq \epsilon.$$
Thus from \mline{9} we have
$$\Delta_1 = \min\left\{\delta_P^{-1},\delta_D^{-1},\alpha \Delta_{0}\right\} \ge \min\left\{\epsilon^{-1},\epsilon^{-1},\alpha\right\} \ge \tilde{\epsilon}^{-1}$$
and hence claim (\ref{eq:iteratea}) for the first iteration.

Assuming (\ref{eq:iterate}) holds for $k$ we know that $\delta_{P,k},\delta_{D,k} \leq \tilde{\epsilon}^k$ and $\Delta_k \geq 1/\tilde{\epsilon}^{k}$. With the scaling factor $\Delta_{k}$ using $x^*=x_k$ and $y^*=y_k$ we scale the QP (\ref{P}) as in Theorem~\ref{th:refinement} and hand it to the QP solver. By Theorem~\ref{th:refinement} this scaled QP is still primal and dual feasible and by Assumption~\ref{assumption} the solver hands back a solution $(\hat{x},\hat{y})$ with tolerance $\epsilon \leq \tilde{\epsilon}$. Therefore using Theorem~\ref{th:refinement} again the next refined iterate ($x_{k+1},y_{k+1}$) has a tolerance in QP (\ref{P}) of $\tilde{\epsilon}/\Delta_{k} \leq \tilde{\epsilon}^{k+1}$, which proves (\ref{eq:iterateb}--\ref{eq:iterated}).

With the same argument the solution $(\hat{x},\hat{y})$ violates complementary slackness by $\sigma$ in the scaled QP (Assumption~\ref{assumption}) and the refined iterate ($x_{k+1},y_{k+1}$) violates complementary slackness in QP (\ref{P}) by $\sigma/\Delta_{k}^2\leq\sigma\tilde{\epsilon}^{2k}$ proving (\ref{eq:iteratee}).

We have now $\delta_{P,k+1},\delta_{D,k+1} \leq \tilde{\epsilon}^{k+1}$. Also it holds that $\alpha\Delta_k \geq \alpha/\tilde{\epsilon}^{k}\geq 1/\tilde{\epsilon}^{k+1}$. \mline{9} of Algorithm~\ref{IQPR} gives
\begin{equation*}
\Delta_{k+1} \geq 1/\tilde{\epsilon}^{k+1},
\end{equation*}
proving (\ref{eq:iteratea}).

Then (\ref{eq:iterationnumber}) follows by assuming the slowest convergence rate of the primal, dual and complementary violations and by comparing this with the termination condition in \mline{11} of Algorithm~\ref{IQPR}
\begin{equation*}
\tilde{\epsilon}^k \leq \epsilon_P,\;\tilde{\epsilon}^k \leq \epsilon_D, \sigma\tilde{\epsilon}^{2(k-1)}\leq\epsilon_S.
\end{equation*}
This is equivalent to (\ref{eq:iterationnumber}). \end{proof}
The results show that even though we did not use the violation of the complementary slackness to choose the scaling factor in Algorithm~\ref{IQPR}, the complementary slackness violation is bounded by the square of $\tilde{\epsilon}$.

\begin{myremark}[Nonconvex QPs]
Algorithm~\ref{IQPR} can also be used to calculate high precision KKT pairs of nonconvex QPs. If the black box QP solver hands back local solutions of the quality specified in Assumption~\ref{assumption} Lemma~\ref{le:rightside} holds as well for nonconvex QPs and Algorithm~\ref{IQPR} returns a high precision local solution.
\end{myremark}

However, assuming strict convexity, an even stronger result holds.
Inspired by the result for the equality-constrained QP~\cite[Proposition 2.12]{Dostal2009} we investigate how this right-hand side convergence of the KKT conditions is related to the primal-dual solution.
\begin{mylemma}[Primal and dual solution accuracy] \label{le:solutionaccuracy}
Let QP (\ref{P}) be given and be strictly convex, the minimal and maximal eigenvalues of $Q$ be $\lambda_{\min}(Q)$ and $\lambda_{\max}(Q)$, respectively, and the minimal nonzero singular value of $A$ be $\sigma_{\min}(A)$. 
Let the KKT conditions (\ref{eq:kkt}) hold for $(x^*,y^*,z^*)$, i.e., 
\begin{subequations}\label{eq:fullkkt}
\begin{eqnarray}
 A x^* & = & b \label{eq:fullkkta}\\
 A^Ty^* +z^* & = & Qx^* + c \label{eq:fullkktb}\\
 z^{*T}(x^*-l) & = & 0 \label{eq:fullkktc}\\
 x^*  & \geq & l \label{eq:fullkktd}\\
 z^* & \geq & 0
\end{eqnarray}
\end{subequations}
and the disturbed KKT conditions for disturbances $e \in \mathbb{Q}^m,\,g,f,h \in \mathbb{Q}^n,$ and $i \in \mathbb{Q}$ hold for $(x,y,z)$, i.e.,
\begin{subequations}\label{eq:fullkkt2}
\begin{eqnarray}
 A x & = & b + e \label{eq:fullkkt2a}\\
 A^Ty +z & = & Qx + c + g \label{eq:fullkkt2b}\\
 z^T(x-l) & = & i \label{eq:fullkkt2c}\\
 x  & \geq & l + f \label{eq:fullkkt2d}\\
 z & \geq & h.
\end{eqnarray}
\end{subequations}
Denote 
\begin{eqnarray*}
a &:=& \frac{ \lambda_{\max}(Q) \| e\|_2}{2\sigma_{\min}(A)} +\lambda_{\max}(Q) \lambda_{\min}(Q)\| g\|_2/2 \\
d &:=&  \lambda_{\max}(Q)\|i- h^T(x^*-l) -z^{*T}f\|_2.
\end{eqnarray*}

Then
\begin{equation} \label{eq:dual}
 \|A^T(y-y^*) + (z-z^*)\|_2 \leq  a  +\sqrt{ a^2 + d }
\end{equation}
and
\begin{equation} \label{eq:primal}
 \|x-x^*\|_2 \leq \lambda_{\min}(Q) ( a  +\sqrt{ a^2 + d } ) + \lambda_{\min}(Q)\|g\|_2
\end{equation}
\end{mylemma}

\begin{proof}
By (\ref{eq:fullkkta}) and (\ref{eq:fullkkt2a}) we have that $A(x-x^*)=e$ and taking the Moore-Penrose pseudoinverse $A^{+}$ of $A$ we define $\delta=A^{+}e$ with $A\delta=e$ and $\|\delta\|_2\leq \sigma_{\min}(A)^{-1}\|e\|_2$. Using this we can start to derive the dual bound by taking the difference of (\ref{eq:fullkktb}) and (\ref{eq:fullkkt2b})
\begin{equation} \label{eq:base}
 A^T(y-y^*) + (z-z^*) = Q(x-x^*) + g.
\end{equation}
Multiplying from the left with $Q^{-1}(A^T(y-y^*) + (z-z^*))$ transposed gives
\begin{equation*}
 \|A^T(y-y^*) + (z-z^*)\|^2_{Q^{-1}} =  (A^T(y-y^*) + (z-z^*))^T  ((x-x^*) + Q^{-1} g).
\end{equation*}
\begin{equation*}
 = (A^T(y-y^*) + (z-z^*))^T  Q^{-1} g + (y-y^*)^T\underbrace{A(x-x^*)}_{A\delta} + (z-z^*)^T(x-x^*). 
\end{equation*}
\begin{equation} \label{eq:rhs2}
 = (A^T(y-y^*) + (z-z^*))^T  (Q^{-1} g + \delta) + (z-z^*)^T(x-x^*-\delta). 
\end{equation}
The second term of (\ref{eq:rhs2}) can be expressed as
\begin{equation*}
 (z-z^*)^T(x-l-(x^*-l)-\delta)=\underbrace{z^T(x-l)}_i+\underbrace{z^{*T}(x^*-l)}_0-\underbrace{z^T(x^*-l)}_{\geq h^T(x^*-l)}-\underbrace{z^{*T}(x-l)}_{\geq z^{*T}f}
\end{equation*}
\begin{equation*}
 (z-z^*)^T(x-l-(x^*-l)-\delta) \leq i - h^T(x^*-l) -z^{*T}f.
\end{equation*}
With this and (\ref{eq:rhs2}) we  bound from above the term $\|A^T(y-y^*) + (z-z^*)\|^2_{Q^{-1}}=*$ giving the inequality
\begin{equation*}
  * \leq (A^T(y-y^*) + (z-z^*))^T  (Q^{-1} g + \delta) + i- h^T(x^*-l) -z^{*T}f.
\end{equation*}
Taking the norm on the right and reordering terms gives
\begin{eqnarray*}
  \|Q\|_2^{-1}\|A^T(y-y^*) + (z-z^*)\|_2^2 \leq \|A^T(y-y^*) + (z-z^*)\|_2  \|Q^{-1} g + \delta\|_2 \\
+ \|i- h^T(x^*-l) -z^{*T}f\|_2.
\end{eqnarray*}
This is a quadratic expression in $\|A^T(y-y^*) + (z-z^*)\|_2=m$
\begin{equation*}
 m^2 - m   \|Q^{-1} g + \delta\|_2 \|Q\|_2 - \|i- h^T(x^*-l) -z^{*T}f\|_2 \|Q\|_2 \leq 0.
\end{equation*}
It has two roots, but only one is greater than zero and bounds $\|A^T(y-y^*) + (z-z^*)\|_2(=m)$ from above
\begin{equation}
\begin{array}{rl}
 m \leq &  \|Q^{-1} g + \delta\|_2 \|Q\|_2/2 \\
&+ \sqrt{ (\|Q^{-1} g + \delta\|_2\|Q\|_2)^2/4 + \|i- h^T(x^*-l) -z^{*T}f\|_2\|Q\|_2 }.
\end{array}
\end{equation}
This can be expressed as
\begin{equation} \label{eq:dualbound}
 \|A^T(y-y^*) + (z-z^*)\|_2 \leq  a  +\sqrt{ a^2 + d }
\end{equation}
where $a$ and $d$ are defined as above. This proves (\ref{eq:dual}). To prove the primal bound we multiply equation (\ref{eq:base}) from the left with $Q^{-1}$
\begin{equation*}
 (x-x^*) = Q^{-1}(A^T(y-y^*)+(z-z^*)-g).
\end{equation*}
Taking norms gives the inequality
\begin{equation} \label{eq:primalbound1}
 \|x-x^*\|_2 \leq \|Q^{-1}\|_2 \|A^T(y-y^*)+(z-z^*)\|_2 + \|Q^{-1}g\|_2.
\end{equation}
Combining the dual bound (\ref{eq:dualbound}) and (\ref{eq:primalbound1}) we get the final primal bound
\begin{equation*} 
 \|x-x^*\|_2 \leq \lambda_{\min}(Q) ( a  +\sqrt{ a^2 + d } ) + \lambda_{\min}(Q) \|g\|_2
\end{equation*}
which proves (\ref{eq:primal}). \end{proof}
Note that $\lambda_{\max}(Q) \lambda_{\min}(Q)$ is the condition number of $Q$. The above assumption and lemmas can be summarized to a statement about the convergence of the algorithm for a strictly convex QP.
\begin{theorem}[Rate of convergence]
Algorithm~\ref{IQPR} with corresponding input and using a QP solver satisfying Assumption~\ref{assumption} solving the QP (\ref{P}) that is also strictly convex has a linear rate of convergence with a factor of $\tilde{\epsilon}^{1/2}$ for the primal iterates, i.e.
\begin{equation*} 
 \|x_k-x^*\| \leq \tilde{\epsilon}^{1/2} \|x_{k-1}-x^*\|,
\end{equation*}
with $x^*$ being the unique solution of (\ref{P}).
\end{theorem}
\begin{proof}
By Assumption~\ref{assumption} and Lemma~\ref{le:rightside} we know that the right-hand side errors of the KKT conditions are bounded by
\begin{equation*} 
 \|e\| \leq \tilde{\epsilon}^{k},\,\|g\| \leq \tilde{\epsilon}^{k},\,\|f\| \leq \tilde{\epsilon}^{k},\,\|i\| \leq \sigma \tilde{\epsilon}^{2(k-1)}, \|h\| = 0.
\end{equation*}
Here we set the violations $h$ of the inequality KKT multipliers $z$ to zero and count them as additional dual violations $g$ for simplicity. Also note that in Lemma~\ref{le:solutionaccuracy} the bound is just depending on the norm of the right-hand side violation vectors, two different violation vectors with the same norm give the same bound. Therefore we just consider the norms. Combining the above with Lemma~\ref{le:solutionaccuracy} we get
\begin{equation*} 
 \|x_k-x^*\| \leq c_1 \tilde{\epsilon}^{k} + \sqrt{c_2\tilde{\epsilon}^{k}+c_3\tilde{\epsilon}^{2k}}
\end{equation*}
for the primal iterate in iteration $k$ with constants
\begin{equation*}
\begin{array}{rcl}
 c_1 & = & \lambda_{\min}(Q) \lambda_{\max}(Q) \left( (\frac{1}{\lambda_{\max}(Q)} + \frac{\lambda_{\min}(Q)}{2}) + \frac{1}{(2\sigma_{\min}(A)}\right) \\
 c_2 & = & \lambda_{\max}(Q) \| z^* \| \\
 c_3 & = & (c_1-\lambda_{min}(Q))^2 + \lambda_{\max}(Q)\sigma/\tilde{\epsilon}^2. \\
\end{array}
\end{equation*}
Looking at the quotient
\begin{equation*}
 \frac{\|x_k-x^*\|}{\|x_{k-1}-x^*\|} \leq \frac{c_1 \tilde{\epsilon}^{k} + \sqrt{c_2\tilde{\epsilon}^{k}+c_3\tilde{\epsilon}^{2k}}}
 {c_1 \tilde{\epsilon}^{k-1} + \sqrt{c_2\tilde{\epsilon}^{k-1}+c_3\tilde{\epsilon}^{2(k-1)}}}
\end{equation*}
and seeing that
\begin{equation*}
\frac{\|x_k-x^*\|}{\|x_{k-1}-x^*\|} \leq \frac{\tilde{\epsilon}^{k/2}(c_1 \tilde{\epsilon}^{k/2} + \sqrt{c_2+c_3\tilde{\epsilon}^{k}})}
 {\tilde{\epsilon}^{(k-1)/2}(c_1 \tilde{\epsilon}^{(k-1)/2} + \sqrt{c_2+c_3\tilde{\epsilon}^{k-1}})}=\tilde{\epsilon}^{1/2}\gamma_k
\end{equation*}
with $\gamma_k \leq 1$ proves the result. \end{proof}
This theoretical investigation shows us two things. First, we have linear residual convergence with a rate of $\tilde{\epsilon}$. In contrast to usual convergence results our algorithm achieves this rate in practice by the use of rational computations if the floating-point solver delivers solutions of the quality specified in Assumption~\ref{assumption}. This is also checked by the rational residual calculation in our algorithm in every iteration. Second, this residual convergence implies primal iterate convergence with a linear rate of $\tilde{\epsilon}^{1/2}$ for strictly convex QPs.

\section{Implementation}
\label{sec:implementation}
Following previous work~\cite{gleixner2012} on the LP case we implemented Algorithm~\ref{IQPR} in the same framework within the \texttt{SoPlex} solver~\cite{wunderling1996paralleler}, version 2.2.1.2, using the GNU multiple precision library (GMP)~\cite{granlund2015gnu} for rational computations, version 6.1.0. As underlying QP solver we use the active-set solver \texttt{qpOASES}~\cite{ferreau2014qpoases} version 3.2. This version of \texttt{qpOASES} was originally designed for small to medium QPs (up to 1,000 variables and constraints).  Furthermore, we implemented an interface to a pre-release version of \texttt{qpOASES} 4.0, which can handle larger, sparse QPs of a size up to 40,000 variables and constraints.  Compared to the matured \texttt{qpOASES}~3.2, this version is not yet capable of hotstarts and in some cases less robust. Nevertheless, it allows us to study the viability of iterative refinement on larger QPs.  The source code of our implementation is available for download in a public repository.\footnote{https://github.com/TobiasWeber/QPrefinement}

In order to treat general QPs with inequalities, our implementation recovers the form (\ref{P}) by adding one slack variable per inequality constraint. Note that not only lower, but also upper bounds on the variables need to be considered. However, this is a straightforward modification to our algorithm and realized in the implementation.

One advantage of using the active-set QP solver \texttt{qpOASES} is the returned basis information.  We use the basis in three aspects: first, to calculate dual and complementary slack violations; second, to explicitly set nonbasic variables to their lower bounds after the refinement step in \mline{16} of Algorithm~\ref{IQPR}; and third, to compute a rational solution defined by the corresponding system of linear equations.  This is solved by a standard LU factorization in rational arithmetic.  If the resulting primal-dual solution is verified to be primal and dual feasible, the algorithm can terminate early with an exact optimal basic solution.

Since the LU factorization can be computationally expensive, we only perform this step if we believe the basis to be optimal.
When the QP solver returns the same basis as ``optimal'' for several iterations this can be used as a heuristic indicator that the basis might be truly optimal, even if the iteratively corrected numerical solution is not yet exact.  Hence, the number of consecutive iterations with the same basis is used to trigger a rational basis system solve.  This can be controlled by a threshold parameter called ``ratfac minstalls'', see Table~\ref{tab:IQPRsettings}.

If the floating-point solver fails to compute an approximately optimal solution, we decrease the scaling factor by two orders of magnitude and try to solve the resulting QP again. The scaling factor is reduced either until the maximum number of backstepping rounds is reached or until the next backstepping round would result in a scaling factor lower than in the last refinement iteration ($k-1$).

Currently, the implementation has no detection mechanism for infeasible or unbounded QPs.  Because the testset at hand contains only convex and feasible QPs, this does not affect the numerical experiments.
The default parameter set (s1) of our implementation is given in Table~\ref{tab:IQPRsettings}. The other four parameter sets (s2-s5) are used for our numerical experiments to derive either exact or inexact solutions.

\begin{table}
\centering
\caption{IQPR parameters}
\label{tab:IQPRsettings}       \smallskip

\begin{tabular}{ll|lll|l}
\toprule
\rowcolor[gray]{0.85}
parameter set & s1 & s2 & s3 & s4 & s5  \\
\rowcolor[gray]{0.85}
\texttt{qpOASES} version & 3.2 & 3.2 & 4.0 & 4.0 & 3.2 \\
\midrule
primal tolerance ($\epsilon_P$) & 1e-100 & 1e-100 & 1e-100 & 1e-100 & 1e-10 \\
dual tolerance ($\epsilon_D$)& 1e-100 & 1e-100 & 1e-100 & 1e-100 & 1e-10 \\
maxscaleincrement ($\alpha$) & 1e12 & 1e12 & 1e12 & 1e12 & 1e12 \\
sparse & no & no & yes & yes & no \\
max num backstepping ($l_{max}$) & 10 & 10 & 10 & 10 & 1 \\
refinement limit ($k_{max}$) & 300 & 50 & 50 & 50 & 10 \\
ratfac minstalls & 2 & 0 & 0 & 51 & 30\\
\bottomrule
\end{tabular}
\end{table}

We exploit the different features of the two \texttt{qpOASES} versions. Version 3.2 has hotstart capabilities that allow to reuse the internal basis system factorization of the preceding optimal basis. Therefore we start in the old optimal basis and build on the progress made in the previous iterations instead of solving the QP from scratch at every iteration. Additionally we increase the termination tolerance and relax other parameters that ensure a reliable solve. This speeds up the solving process and is possible because the inaccuracies, introduced by this in the floating-point solution, are detected anyway and handed back to the QP solver in the next iteration for correction. If the QP solver fails we simply change to reliable settings and resolve the same QP from the same starting basis before downscaling. Hence, in Algorithm~\ref{IQPR} each `solve' statement means: try fast settings first and if this fails switch to slow and reliable settings of \texttt{qpOASES} 3.2. These two sets of options are given in Table~\ref{tab:settings}. 
For the pre-release version 4.0 we use default settings and no resolves. We either factorize after each iteration or not at all (see Table~\ref{tab:IQPRsettings}).

\begin{table}[t!]
  \centering
\caption{\texttt{qpOASES} options (version 3.2)}
\label{tab:settings}       \smallskip

\begin{tabular}{lll}
\toprule
\rowcolor[gray]{0.85}
Option & Fast & Reliable  \\
\midrule
Standard settings set & MPC & Reliable \\
NZCTests & enabled & enabled (default) \\
DriftCorrection & enabled & enabled (default) \\
Ramping & enabled & enabled (default) \\
terminationTolerance & 1e-3 & 1.1105e-9 (default) \\
numRefinementSteps & 0 (default) & 10 \\
enableFullLITests & 0 & 0\\
\bottomrule
\end{tabular}
\end{table}

\section{Numerical results}
\label{sec:NumericalExperiments}
For the numerical experiments the standard testset of Maros and M{\'e}sz{\'a}ros~\cite{maros1999repository} is used. It contains 138 convex QPs that feature between two and about 90,000 variables. The number of constraints varies from one to about 180,000 and the number of nonzeros ranges between two and about 550,000.

We perform two different experiments. The goal of the first experiment is to solve as many QPs from the testset as precisely as possible in order to analyze the iterative refinement procedure computationally and to provide exact reference solutions for future research on QP solvers. In the second experiment we want to compare \texttt{qpOASES} (version 3.2, no QP refinement, one solve, default settings) to low accuracy refinement (low tolerance of 1e-10 in Algorithm~\ref{IQPR}, using also \texttt{qpOASES} 3.2). This allows us to investigate whether refinement could also be beneficial in cases that do not require extremely high accuracy, but a strictly guaranteed solution tolerance in shortest possible runtime.

\paragraph{Experiment 1}

We use the three different parameter sets (s2-s4) given in Table~\ref{tab:IQPRsettings} to calculate exact solutions. The first set (s2) contains a primal and dual termination tolerance of 1e-100, enables rational factorization in every iteration, and allows for 50 refinements and 10 backsteppings using a dense QP formulation with \texttt{qpOASES} version 3.2. In contrast the other two sets (s3, s4) with \texttt{qpOASES} version 4.0 use a sparse QP formulation, either with factorization in every iteration or without factorization.

Table~\ref{tab:resultsexactsmall} states for each setting the number of instances which it solved exactly, for which tolerance 1e-100 was reached, and for which it failed to produce a high-precision solution.  In total these three strategies could solve 91 out of the 138~QPs in the testset \emph{exactly} and 39~instances within tolerance 1e-100.  For eight instances no high-precision solution was computed. These ``virtual best'' results stated in the fifth column consider for each QP the result of the individual parameter sets that resulted in the smallest violation. It should be emphasized that for each of the three parameter sets there exists at least one instance for which it produced the most accurate solution.

The last column reports the average number of nonzeros of the QPs in the three ``virtual best'' categories. This suggests that for problems with less nonzeros a higher accuracy was reached. Furthermore, one sees that the entries in the second column (set s2) do not sum to 138 because with the dense QP formulation some of the problems do not return from the cluster due to memory limitations. Hence the set s2 fails in total on 32~instances. For the parameter set s4 without rational factorization we see that one QP is solved by chance exactly while for all others the algorithm terminates with violations greater zero.

In order to solve the 197 (=33+45+118) QPs to high precision the algorithm needed on average 8.84 refinements. This confirms the linear convergence because we bounded the increase of the scaling factor in each iteration by $\alpha = 10^{12}$ and terminate after reaching a tolerance of $10^{-100}$. If \texttt{qpOASES} would consistently return solutions with an accuracy of $10^{-12}$ we would expect the algorithm to need 9 iterations ($100/12 \approx 8.33\ldots$ rounded up). We see that \texttt{qpOASES} usually delivers solutions of a tolerance below $10^{-12}$.

Detailed results can be found in the Appendix in Tables~\ref{tab:resultsexactlarge}, \ref{tab:resultsexactqorelarge}, and \ref{tab:resultsexactqorenofaclarge}. If an exact solution is found \texttt{qpOASES} usually returns the optimal basis in the first three iterations (refinements).  Subsequently, the corresponding basis system is solved exactly by a rational LU factorization. For six problems we found that the objective values given in~\cite{maros1999repository} differ from our results by more than \mbox{1e-7}: GOULDQP2, HS268, S268, HUESTIS, HUES-MOD, and LISWET8. This might be due to the use of a floating-point QP solver with termination tolerance about \mbox{1e-7} when originally computing the values reported.  The precise objective value can be found in the online material associated with this paper.

\begin{table}[t!]
  \centering
  \caption{Results for the three exact parameter sets (s2-s4) over all 138~QPs in the testset: number of instances according to terminal solution accuracy for each setting, for the virtual best setting, and the average number of nonzeros over the instances in the ``best'' categories.}
  \label{tab:resultsexactsmall}

\clearpage{}\begin{tabular}{lrrr||rr}
\toprule
\rowcolor[gray]{0.85}
Accuracy reached & \;\;\;s2 & \;\;\;s3 & \;\;\;s4 & best & avg.~nnzs\\\midrule
Exact (no viol.) & 73 & 74 & 1 & 91 & 6.76e+03 \\
High ($\leq$ 1e-100) & 33 & 45 & 118 & 39 & 1.45e+04 \\
Low ($>$ 1e-100) & 11 & 18 & 18 & 8 & 9.34e+04 \\
Fail (Not returned) & 21 & 1 & 1 & 0 & \\
\bottomrule
\end{tabular}\clearpage{}

\end{table}

\paragraph{Experiment 2}

In the following the iterative refinement algorithm is set to a termination tolerance $10^{-10}$ and the rational factorization of the basis system is disabled. The refinement limit is set to 10 and the backstepping limit is set to one (parameter set s5). We compare this implementation to \texttt{qpOASES}~3.2 with the three predefined \texttt{qpOASES} settings (MPC, Default, Reliable) that include termination tolerances of 2.2210e-7, 1.1105e-9, and 1.1105e-9, respectively. For these fast solves we select only part of the testset, including the 73 problems that have no more than 1,000 variables and constraints. This corresponds to the problem sizes for which \texttt{qpOASES}~3.2 was originally designed. In order to allow for a meaningful comparison of runtimes, the evaluation only considers QPs which were solved by all three \texttt{qpOASES} 3.2 settings and by refinement to ``optimality'', where optimality was certified by \texttt{qpOASES} 3.2 (with its internal floating-point checks) or rational checks in our algorithm, respectively.

\begin{table}[t!]
  \centering
  \caption{Performance comparison for inexact solves (runtimes are in seconds).}
  \label{tab:resultsinexactsmall}

\clearpage{}\begin{tabular}{llrlrrr}
\toprule
\rowcolor[gray]{0.85}
&  & \multicolumn{2}{c}{IQPR s5} & \multicolumn{3}{c}{\texttt{qpOASES} with standard settings:}  \\
\rowcolor[gray]{0.85}
Measure & subset & & & MPC & Default & Reliable \\
\midrule
\multirow{2}{*}{\shortstack[l]{Time: arith. mean \\ (\% rat. time)}} & all & 2.54 & (0.16) & 1.03 & 2.77 & 19.58 \\
& $>0.01$ & 3.25 & (0.16) & 1.32 & 3.55 & 25.08 \\
& $>0.1$ & 4.02 & (0.12) & 1.64 & 4.40 & 31.07 \\
& $>1$ & 5.66 & (0.12) & 2.31 & 6.27 & 44.60 \\
& $>10$ & 7.19 & (0.11) & 2.77 & 9.06 & 69.39 \\
\midrule
\multirow{2}{*}{\shortstack[l]{Time: shifted geo.\\ mean, shift=0.01\\ (\% rat. time)}} & all & 0.16 & (1.68) & 0.08 & 0.10 & 0.16 \\
& $>0.01$ & 0.36 & (0.98) & 0.15 & 0.20 & 0.36 \\
& $>0.1$ & 0.60 & (0.54) & 0.24 & 0.32 & 0.64 \\
& $>1$ & 0.94 & (0.49) & 0.43 & 0.67 & 1.66 \\
& $>10$ & 0.54 & (1.00) & 0.26 & 0.51 & 1.51 \\
\midrule
\multirow{2}{*}{\shortstack[l]{QP solver iterations: \\arith. mean}} & all & 283.75 & & 260.53 & 389.92 & 386.96 \\
& $>0.01$ & 362.16 & & 332.44 & 496.91 & 493.09 \\
& $>0.1$ & 436.67 & & 400.96 & 591.35 & 586.96 \\
& $>1$ & 520.91 & & 479.38 & 765.59 & 761.56 \\
& $>10$ & 353.00 & & 348.25 & 837.15 & 832.75 \\
\midrule
\multirow{2}{*}{\shortstack[l]{QP solver iterations: \\ shifted geo. mean,\\shift=1}} & all & 38.43 & & 36.86 & 62.08 & 61.68 \\
& $>0.01$ & 85.18 & & 80.86 & 113.98 & 112.72 \\
& $>0.1$ & 108.12 & & 101.95 & 124.46 & 123.05 \\
& $>1$ & 105.41 & & 100.22 & 144.20 & 142.88 \\
& $>10$ & 36.21 & & 35.23 & 69.62 & 69.39 \\
\midrule
\multirow{2}{*}{\shortstack[l]{Tolerance: \\ arith. mean}} & all & 1.49e-12 & & 1.29e-08 & 1.10e-08 & 2.28e-09 \\
& $>0.01$ & 1.91e-12 & & 1.65e-08 & 1.40e-08 & 2.92e-09 \\
& $>0.1$ & 2.10e-12 & & 2.03e-08 & 1.74e-08 & 3.62e-09 \\
& $>1$ & 8.89e-13 & & 2.21e-08 & 2.48e-08 & 4.98e-09 \\
& $>10$ & 5.29e-14 & & 1.42e-08 & 3.92e-08 & 7.32e-09 \\
\midrule
\multirow{2}{*}{\shortstack[l]{Tolcerance: \\ shifted geo. mean,\\shift=1e-20}} & all & 1.29e-16 & & 2.14e-12 & 1.62e-15 & 4.34e-15 \\
& $>0.01$ & 8.71e-17 & & 1.00e-11 & 4.91e-15 & 1.13e-14 \\
& $>0.1$ & 3.57e-17 & & 8.45e-12 & 6.92e-15 & 2.10e-14 \\
& $>1$ & 1.94e-17 & & 1.08e-12 & 2.44e-15 & 6.02e-15 \\
& $>10$ & 9.36e-19 & & 6.86e-15 & 3.21e-16 & 2.45e-16 \\
\bottomrule
\end{tabular}\clearpage{}

\end{table}

An overview over the performance results is given in Table~\ref{tab:resultsinexactsmall}. We report runtime, QP solver iterations, and the final tolerance reached, each time as arithmetic and shifted geometric mean.  To facilitate a more detailed analysis, we consider the series of subsets ``$> t$'' of instances, for which at least one algorithm took more than $t$~seconds. Equivalently, we exclude the QPs for which all settings took at most 0.01\,s, 0.1\,s, 1\,s, and 10\,s seconds.  Defining the exclusion by all instead of one method only avoids a biased definition of these sets of increasing difficulty.

The results show that in no case the mean runtime of the refinement algorithm is larger than the runtime of \texttt{qpOASES} with reliable setting.  At the same time, the accuracy reached is always significantly higher.  Compared to \texttt{qpOASES} Default, which results in an even lower level of precision, refinement is faster in arithmetic and slightly slower in shifted geometric mean. The QP solver iterations of the refinement are comparable to the MPC setting. When looking at the different subsets we see that for QPs with larger runtime the refinement approach performs relatively better (smaller runtime, iterations and lower tolerance) than the three \texttt{qpOASES} 3.2 standard settings. The refinement guarantees the tolerance of 1e-10 if it does not fail. To achieve this tolerance, for 9 QPs two refinements were necessary, for 21 QPs only one refinement was necessary, and for 35 instances no refinement was necessary at all. The rational computation overhead stated in brackets after the runtime and is well below 2\%. The details are shown in Table~\ref{tab:resultsinexactlarge} in the Appendix. Also note that due to exclusion of fails (which mainly occur with the \texttt{qpOASES} MPC settings) the summarized results have a slight bias towards \texttt{qpOASES}.

\section{Conclusion}
\label{sec:Conclusion}
We presented a novel refinement algorithm and proved linear convergence of residuals and errors. Notably, this theoretical convergence result also carries over to our implementation due to the use of exact rational calculations. We provided high-precision solutions for most of the QPs in the Maros and M{\'e}sz{\'a}ros testset, correcting inaccuracies in optimal solution values reported in the literature. This is beneficial for future research on QP solvers that are evaluated on this testset.

In a second experiment we saw that iterative refinement provides proven tolerance solutions with smaller or equal computation times compared to \texttt{qpOASES} with ``Reliable'' solver settings. It can therefore be used as tool to increase the reliability and speed of standard floating-point QP solvers.
If optimal solutions are needed for rigorous reasoning or to make decisions in the real world the algorithm presented is useful because it is able to fully ensure a specified tolerance. This tolerance then can be adapted to the necessity of the application at hand. At the same time this comes with little overhead in rational computation time, which is important for practical applications.

Regarding algorithmic research and solver development, our framework also provides the possibility to compare different floating-point QP solvers by looking at the number of refinements needed with each solver to detect optimal bases or solutions of a specified tolerance as a measure for solver accuracy. Solver robustness can be checked precisely because violations are computed in rational precision.
In the future, the implementation should be extended for the detection of unbounded or infeasible QPs. Also one could try more general variable transformations, e.g. having a different scaling factor for each variable.
As a concluding remark, we hope that the idea of checking numerical results of floating-point algorithms in exact or safe arithmetic will become a future trend when applying or analyzing numerical algorithms.

\renewcommand{\refname}{\normalsize References}
\setlength{\bibsep}{0.25ex plus 0.3ex}
\bibliographystyle{abbrvnat}

\begin{small}
\bibliography{template}

\end{small}

\clearpage
\section*{Appendix}

{\small
\clearpage{}\begin{longtable}[c]{l|l|l|l|l|l|l|l}
\caption{Detailed results for exact solve of large QP set with parameter set~s2 (Iter.=Iterations, Tol.=Tolerance, Ref.=Refinements, Back.=Backstepping, Res.=Resolves)}
\label{tab:resultsexactlarge}\\
\toprule
\rowcolor[gray]{0.85}
QP Name & Status & Time & Iter. & Tol. & Ref. & Back. & Res. \\
\rowcolor[gray]{0.85}
        &        & [s]  & [\#]     & [-]       & [\#]      & [\#]   & [\#]  \\
\midrule
\endfirsthead
\toprule
\rowcolor[gray]{0.85}
QP Name & Status & Time & Iter. & Tol. & Ref. & Back. & Res. \\
\rowcolor[gray]{0.85}
        &        & [s]  & [\#]     & [-]       & [\#]      & [\#]   & [\#]  \\
\midrule
\endhead

\midrule
\multicolumn{8}{r}{continued on next page}\\
\bottomrule
\endfoot

\bottomrule
\endlastfoot
AUG3D & optimal & 3492.32 & 0 & 2.56e-113 & 8 & 0 & 0 \\
AUG3DC & optimal & 219.36 & 0 &  00 & 0 & 0 & 0 \\
AUG3DCQP & optimal & 218.64 & 540 &  00 & 1 & 0 & 0 \\
AUG3DQP & optimal & 3092.52 & 324 & 2.15e-110 & 8 & 0 & 0 \\
BOYD1 & abort & NaN & NaN & NaN & NaN & NaN & NaN \\
BOYD2 & abort & NaN & NaN & NaN & NaN & NaN & NaN \\
CONT-050 & optimal & 4306.67 & 1 &  00 & 0 & 0 & 0 \\
CONT-100 & timeout & 10814.69 & 3 & NaN & 7 & 6 & 0 \\
CONT-101 & timeout & 10830.68 & 2 & NaN & 7 & 6 & 0 \\
CONT-300 & abort & NaN & NaN & NaN & NaN & NaN & NaN \\
CVXQP1\_L & fail & 15786.74 & 4010 & 55500 & 0 & 0 & 0 \\
CVXQP1\_M & optimal & 16.78 & 412 &  00 & 0 & 0 & 0 \\
CVXQP1\_S & optimal & 0.04 & 36 &  00 & 0 & 0 & 0 \\
CVXQP2\_M & optimal & 25.97 & 651 &  00 & 0 & 0 & 0 \\
CVXQP2\_S & optimal & 0.01 & 62 &  00 & 0 & 0 & 0 \\
CVXQP3\_L & timeout & 10803.72 & 2990 & NaN & 1 & 0 & 0 \\
CVXQP3\_M & optimal & 5.95 & 231 &  00 & 1 & 0 & 0 \\
CVXQP3\_S & optimal & 0.00 & 22 &  00 & 0 & 0 & 0 \\
DPKLO1 & optimal & 2.23 & 0 &  00 & 0 & 0 & 0 \\
DTOC3 & fail & 19121.42 & 0 & 4.40e-15 & 0 & 0 & 0 \\
DUAL1 & optimal & 0.26 & 26 &  00 & 1 & 0 & 0 \\
DUAL2 & optimal & 1.05 & 4 &  00 & 1 & 0 & 0 \\
DUAL3 & optimal & 1.90 & 14 &  00 & 1 & 0 & 0 \\
DUAL4 & optimal & 0.21 & 13 &  00 & 0 & 0 & 0 \\
DUALC1 & optimal & 0.11 & 31 &  00 & 0 & 0 & 0 \\
DUALC2 & optimal & 0.07 & 28 &  00 & 0 & 0 & 0 \\
DUALC5 & optimal & 0.07 & 3 &  00 & 0 & 0 & 0 \\
DUALC8 & optimal & 0.39 & 13 &  00 & 0 & 0 & 0 \\
EXDATA & optimal & 8964.52 & 6738 &  00 & 1 & 0 & 0 \\
GENHS28 & optimal & 0.00 & 0 &  00 & 0 & 0 & 0 \\
GOULDQP2 & inconsistent & 6.70 & 585 &  00 & 2 & 0 & 0 \\
GOULDQP3 & optimal & 1.84 & 176 &  00 & 0 & 0 & 0 \\
HS118 & optimal & 0.00 & 24 &  00 & 0 & 0 & 0 \\
HS21 & optimal & 0.00 & 3 &  00 & 0 & 0 & 0 \\
HS268 & inconsistent & 0.00 & 0 &  00 & 0 & 0 & 0 \\
HS35 & optimal & 0.00 & 1 &  00 & 0 & 0 & 0 \\
HS35MOD & optimal & 0.00 & 0 &  00 & 0 & 0 & 0 \\
HS51 & optimal & 0.00 & 0 &  00 & 0 & 0 & 0 \\
HS52 & optimal & 0.00 & 0 &  00 & 0 & 0 & 0 \\
HS53 & optimal & 0.00 & 0 &  00 & 0 & 0 & 0 \\
HS76 & optimal & 0.00 & 4 &  00 & 0 & 0 & 0 \\
HUESTIS & inconsistent & 9782.42 & 554 &  00 & 0 & 0 & 0 \\
HUES-MOD & inconsistent & 9408.64 & 554 &  00 & 0 & 0 & 0 \\
KSIP & optimal & 31.56 & 1080 &  00 & 1 & 0 & 0 \\
LASER & optimal & 593.20 & 2838 &  00 & 0 & 0 & 0 \\
LOTSCHD & optimal & 0.00 & 5 &  00 & 0 & 0 & 0 \\
MOSARQP1 & optimal & 445.13 & 1528 &  00 & 0 & 0 & 0 \\
MOSARQP2 & optimal & 21.45 & 332 &  00 & 0 & 0 & 0 \\
PRIMAL1 & optimal & 0.51 & 77 &  00 & 0 & 0 & 0 \\
PRIMAL2 & optimal & 2.66 & 94 &  00 & 0 & 0 & 0 \\
PRIMAL3 & optimal & 4.37 & 117 &  00 & 0 & 0 & 0 \\
PRIMAL4 & optimal & 21.04 & 88 &  00 & 0 & 0 & 0 \\
PRIMALC1 & optimal & 0.11 & 234 &  00 & 0 & 0 & 0 \\
PRIMALC2 & optimal & 0.17 & 237 &  00 & 0 & 0 & 0 \\
PRIMALC5 & optimal & 0.23 & 288 &  00 & 0 & 0 & 0 \\
PRIMALC8 & optimal & 1.86 & 515 &  00 & 0 & 0 & 0 \\
Q25FV47 & optimal & 372.89 & 7362 &  00 & 0 & 0 & 0 \\
QADLITTL & optimal & 0.09 & 232 &  00 & 0 & 0 & 0 \\
QAFIRO & optimal & 0.00 & 30 & 6.15e-107 & 7 & 0 & 0 \\
QBANDM & optimal & 2.93 & 1164 &  00 & 0 & 0 & 0 \\
QBEACONF & optimal & 0.91 & 305 & 8.07e-110 & 9 & 0 & 0 \\
QBORE3D & optimal & 0.87 & 536 & 1.97e-111 & 9 & 0 & 0 \\
QBRANDY & optimal & 1.07 & 450 & 4.86e-108 & 8 & 0 & 0 \\
QCAPRI & optimal & 2.95 & 1051 & 1.76e-106 & 8 & 0 & 1 \\
QE226 & optimal & 4.66 & 1396 & 1.71e-101 & 8 & 0 & 0 \\
QETAMACR & optimal & 4.75 & 559 & 5.17e-102 & 9 & 0 & 0 \\
QFFFFF80 & error & 44.78 & 1079 & NaN & 9 & 3 & 5 \\
QFORPLAN & optimal & 2.20 & 1174 &  00 & 0 & 0 & 0 \\
QGFRDXPN & optimal & 29.31 & 1423 &  00 & 0 & 0 & 0 \\
QGROW15 & optimal & 8.11 & 632 & 2.54e-105 & 8 & 0 & 0 \\
QGROW22 & optimal & 32.92 & 944 & 2.86e-109 & 10 & 0 & 0 \\
QGROW7 & optimal & 0.83 & 298 & 7.99e-112 & 8 & 0 & 0 \\
QISRAEL & optimal & 0.38 & 290 &  00 & 0 & 0 & 0 \\
QPCBLEND & optimal & 0.06 & 66 &  00 & 1 & 0 & 0 \\
QPCBOEI1 & optimal & 3.64 & 435 &  00 & 1 & 0 & 0 \\
QPCBOEI2 & optimal & 0.23 & 175 &  00 & 0 & 0 & 0 \\
QPCSTAIR & optimal & 2.76 & 257 &  00 & 0 & 0 & 0 \\
QPILOTNO & optimal & 573.39 & 7442 & 1.99e-102 & 13 & 0 & 0 \\
QPTEST & optimal & 0.00 & 1 &  00 & 0 & 0 & 0 \\
QRECIPE & optimal & 0.07 & 48 & 9.74e-110 & 7 & 0 & 0 \\
QSC205 & optimal & 0.78 & 106 & 1.50e-107 & 8 & 0 & 0 \\
QSCAGR25 & optimal & 6.17 & 1076 &  00 & 0 & 0 & 0 \\
QSCAGR7 & optimal & 0.14 & 329 &  00 & 0 & 0 & 0 \\
QSCFXM1 & optimal & 8.05 & 801 & 1.13e-112 & 11 & 0 & 1 \\
QSCFXM2 & optimal & 81.35 & 1891 & 1.25e-108 & 10 & 0 & 0 \\
QSCFXM3 & optimal & 319.56 & 2510 & 1.77e-107 & 10 & 0 & 0 \\
QSCORPIO & optimal & 2.43 & 233 & 1.08e-101 & 7 & 0 & 0 \\
QSCRS8 & optimal & 49.19 & 2016 &  00 & 1 & 0 & 0 \\
QSCSD1 & optimal & 8.78 & 2054 &  00 & 1 & 0 & 0 \\
QSCSD6 & optimal & 106.56 & 6015 & 4.11e-101 & 9 & 0 & 0 \\
QSCSD8 & optimal & 1560.58 & 19564 & 2.22e-104 & 8 & 0 & 0 \\
QSCTAP1 & optimal & 11.22 & 1474 & 2.63e-103 & 8 & 0 & 0 \\
QSCTAP2 & optimal & 860.50 & 3862 & 6.58e-113 & 10 & 0 & 0 \\
QSCTAP3 & optimal & 2156.04 & 6964 & 2.49e-101 & 8 & 0 & 0 \\
QSEBA & optimal & 26.60 & 1677 &  00 & 0 & 0 & 0 \\
QSHARE1B & optimal & 0.51 & 782 &  00 & 1 & 0 & 0 \\
QSHARE2B & optimal & 0.19 & 359 &  00 & 1 & 0 & 0 \\
QSHELL & optimal & 119.14 & 3306 & 7.69e-109 & 9 & 0 & 1 \\
QSHIP04L & optimal & 398.29 & 4847 & 1.67e-106 & 8 & 0 & 0 \\
QSHIP04S & optimal & 130.24 & 2908 & 2.11e-107 & 8 & 0 & 1 \\
QSHIP08L & optimal & 3711.96 & 7911 & 3.68e-102 & 8 & 0 & 0 \\
QSHIP08S & optimal & 738.69 & 3997 & 1.86e-110 & 9 & 0 & 0 \\
QSHIP12L & optimal & 8386.98 & 9765 & 4.22e-109 & 8 & 0 & 1 \\
QSHIP12S & optimal & 1176.10 & 3821 & 2.30e-109 & 8 & 0 & 0 \\
QSIERRA & error & 996.02 & 6421 & NaN & 0 & 0 & 1 \\
QSTAIR & optimal & 13.43 & 1136 & 7.18e-111 & 9 & 1 & 2 \\
QSTANDAT & optimal & 26.45 & 1222 &  00 & 3 & 0 & 0 \\
S268 & inconsistent & 0.00 & 0 &  00 & 0 & 0 & 0 \\
STADAT1 & optimal & 8013.74 & 9995 &  00 & 0 & 0 & 0 \\
STADAT2 & optimal & 6437.65 & 6527 &  00 & 1 & 0 & 0 \\
STADAT3 & timeout & 18337.54 & 2144 & NaN & 0 & 0 & 1 \\
STCQP1 & optimal & 604.13 & 353 & 1.05e-103 & 7 & 0 & 0 \\
STCQP2 & optimal & 419.60 & 105 &  00 & 0 & 0 & 0 \\
TAME & optimal & 0.00 & 0 &  00 & 0 & 0 & 0 \\
VALUES & optimal & 0.10 & 143 &  00 & 2 & 0 & 1 \\
YAO & optimal & 659.58 & 2001 &  00 & 1 & 0 & 0 \\
ZECEVIC2 & optimal & 0.00 & 3 &  00 & 0 & 0 & 0 \\
\end{longtable}
\clearpage{}
}
{\small
\clearpage{}\begin{longtable}[c]{l|l|l|l|l|l|l|l}
\caption{Detailed results for exact solve of large QP set with parameter set~s3 (Iter.=Iterations, Tol.=Tolerance, Ref.=Refinements, Back.=Backstepping, Res.=Resolves)}
\label{tab:resultsexactqorelarge}\\
\toprule
\rowcolor[gray]{0.85}
QP Name & Status & Time & Iter. & Tol. & Ref. & Back. & Res. \\
\rowcolor[gray]{0.85}
        &        & [s]  & [\#]     & [-]       & [\#]      & [\#]   & [\#]  \\
\midrule
\endfirsthead
\toprule
\rowcolor[gray]{0.85}
QP Name & Status & Time & Iter. & Tol. & Ref. & Back. & Res. \\
\rowcolor[gray]{0.85}
        &        & [s]  & [\#]     & [-]       & [\#]      & [\#]   & [\#]  \\
\midrule
\endhead

\midrule
\multicolumn{8}{r}{continued on next page}\\
\bottomrule
\endfoot

\bottomrule
\endlastfoot
AUG2D & timeout & 10865.56 & 0 & NaN & 1 & 0 & 0 \\
AUG2DC & timeout & 10836.45 & 0 & NaN & 6 & 5 & 0 \\
AUG2DCQP & optimal & 155.71 & 0 & 1.47e-101 & 8 & 0 & 0 \\
AUG2DQP & timeout & 10801.84 & 0 & NaN & 6 & 5 & 0 \\
AUG3D & error & 7.60 & 0 & NaN & 1 & 0 & 0 \\
AUG3DC & optimal & 112.66 & 0 &  00 & 0 & 0 & 0 \\
AUG3DCQP & optimal & 19.50 & 0 &  00 & 0 & 0 & 0 \\
AUG3DQP & optimal & 7.62 & 0 &  00 & 0 & 0 & 0 \\
BOYD2 & abort & NaN & NaN & NaN & NaN & NaN & NaN \\
CONT-050 & optimal & 4224.16 & 0 &  00 & 0 & 0 & 0 \\
CONT-100 & timeout & 10806.77 & 0 & NaN & 7 & 6 & 0 \\
CONT-101 & timeout & 10809.28 & 0 & NaN & 7 & 6 & 0 \\
CONT-200 & timeout & 10802.04 & 0 & NaN & 7 & 6 & 0 \\
CONT-201 & timeout & 10801.56 & 0 & NaN & 7 & 6 & 0 \\
CONT-300 & abort & NaN & NaN & NaN & NaN & NaN & NaN \\
CVXQP1\_L & optimal & 6641.26 & 0 & 1.90e-101 & 9 & 0 & 0 \\
CVXQP1\_M & optimal & 8.02 & 0 &  00 & 0 & 0 & 0 \\
CVXQP1\_S & optimal & 0.02 & 0 &  00 & 0 & 0 & 0 \\
CVXQP2\_L & timeout & 11041.07 & 0 & NaN & 5 & 4 & 0 \\
CVXQP2\_M & optimal & 19.13 & 0 &  00 & 0 & 0 & 0 \\
CVXQP2\_S & optimal & 0.03 & 0 & 2.25e-110 & 7 & 0 & 0 \\
CVXQP3\_L & error & 7587.33 & 0 & NaN & 6 & 1 & 0 \\
CVXQP3\_M & optimal & 2.33 & 0 &  00 & 0 & 0 & 0 \\
CVXQP3\_S & reached & 0.48 & 0 & NaN & 50 & 0 & 0 \\
DPKLO1 & optimal & 2.26 & 0 &  00 & 0 & 0 & 0 \\
DTOC3 & optimal & 2611.08 & 0 &  00 & 0 & 0 & 0 \\
DUAL1 & optimal & 0.15 & 0 &  00 & 0 & 0 & 0 \\
DUAL2 & optimal & 0.53 & 0 &  00 & 0 & 0 & 0 \\
DUAL3 & optimal & 0.78 & 0 &  00 & 0 & 0 & 0 \\
DUAL4 & optimal & 0.20 & 0 &  00 & 0 & 0 & 0 \\
DUALC1 & optimal & 0.07 & 0 &  00 & 0 & 0 & 0 \\
DUALC2 & optimal & 0.05 & 0 &  00 & 0 & 0 & 0 \\
DUALC5 & optimal & 0.08 & 0 &  00 & 0 & 0 & 0 \\
DUALC8 & optimal & 0.19 & 0 &  00 & 0 & 0 & 0 \\
EXDATA & optimal & 1407.82 & 0 &  00 & 0 & 0 & 0 \\
GENHS28 & optimal & 0.00 & 0 &  00 & 0 & 0 & 0 \\
GOULDQP2 & inconsistent & 0.24 & 0 & 1.50e-110 & 7 & 0 & 0 \\
GOULDQP3 & optimal & 0.11 & 0 &  00 & 0 & 0 & 0 \\
HS118 & optimal & 0.00 & 0 &  00 & 0 & 0 & 0 \\
HS21 & optimal & 0.00 & 0 &  00 & 0 & 0 & 0 \\
HS268 & inconsistent & 0.00 & 0 &  00 & 0 & 0 & 0 \\
HS35 & optimal & 0.00 & 0 &  00 & 0 & 0 & 0 \\
HS35MOD & optimal & 0.00 & 0 &  00 & 0 & 0 & 0 \\
HS51 & optimal & 0.00 & 0 & 3.14e-108 & 6 & 0 & 0 \\
HS52 & optimal & 0.00 & 0 &  00 & 0 & 0 & 0 \\
HS53 & optimal & 0.00 & 0 &  00 & 0 & 0 & 0 \\
HS76 & optimal & 0.00 & 0 &  00 & 0 & 0 & 0 \\
HUESTIS & inconsistent & 40.17 & 0 &  00 & 0 & 0 & 0 \\
HUES-MOD & inconsistent & 21.27 & 0 &  00 & 0 & 0 & 0 \\
KSIP & optimal & 2.38 & 0 &  00 & 0 & 0 & 0 \\
LASER & optimal & 41.47 & 0 &  00 & 0 & 0 & 0 \\
LISWET1 & optimal & 57.85 & 0 &  00 & 0 & 0 & 0 \\
LISWET10 & optimal & 645.20 & 0 & 1.60e-105 & 15 & 0 & 0 \\
LISWET11 & optimal & 54.26 & 0 &  00 & 0 & 0 & 0 \\
LISWET12 & optimal & 63.22 & 0 &  00 & 0 & 0 & 0 \\
LISWET2 & optimal & 704.57 & 0 & 4.17e-107 & 17 & 0 & 0 \\
LISWET3 & optimal & 306.45 & 0 & 1.20e-101 & 14 & 0 & 0 \\
LISWET4 & optimal & 317.64 & 0 & 2.60e-107 & 15 & 0 & 0 \\
LISWET5 & optimal & 324.42 & 0 & 4.27e-106 & 15 & 0 & 0 \\
LISWET6 & optimal & 325.27 & 0 & 3.68e-102 & 14 & 0 & 0 \\
LISWET7 & optimal & 54.87 & 0 &  00 & 0 & 0 & 0 \\
LISWET8 & inconsistent & 56.98 & 0 &  00 & 0 & 0 & 0 \\
LISWET9 & optimal & 64.09 & 0 &  00 & 0 & 0 & 0 \\
LOTSCHD & optimal & 0.00 & 0 &  00 & 0 & 0 & 0 \\
MOSARQP1 & optimal & 0.96 & 0 &  00 & 0 & 0 & 0 \\
MOSARQP2 & optimal & 0.52 & 0 &  00 & 0 & 0 & 0 \\
POWELL20 & optimal & 44.65 & 0 &  00 & 0 & 0 & 0 \\
PRIMAL1 & optimal & 0.20 & 0 &  00 & 0 & 0 & 0 \\
PRIMAL2 & optimal & 0.71 & 0 &  00 & 0 & 0 & 0 \\
PRIMAL3 & optimal & 1.45 & 0 &  00 & 0 & 0 & 0 \\
PRIMAL4 & optimal & 0.76 & 0 &  00 & 0 & 0 & 0 \\
PRIMALC1 & optimal & 0.01 & 0 &  00 & 0 & 0 & 0 \\
PRIMALC2 & optimal & 0.01 & 0 &  00 & 0 & 0 & 0 \\
PRIMALC5 & optimal & 0.01 & 0 &  00 & 0 & 0 & 0 \\
PRIMALC8 & optimal & 0.02 & 0 &  00 & 0 & 0 & 0 \\
Q25FV47 & optimal & 9.14 & 0 & 3.75e-114 & 9 & 0 & 0 \\
QADLITTL & optimal & 0.01 & 0 &  00 & 0 & 0 & 0 \\
QAFIRO & optimal & 0.00 & 0 &  00 & 0 & 0 & 0 \\
QBANDM & optimal & 0.20 & 0 & 1.17e-103 & 6 & 0 & 0 \\
QBEACONF & inconsistent & 0.16 & 0 & NaN & 50 & 0 & 0 \\
QBORE3D & inconsistent & 0.18 & 0 & NaN & 50 & 0 & 0 \\
QBRANDY & optimal & 0.18 & 0 & 4.04e-106 & 9 & 0 & 0 \\
QCAPRI & optimal & 0.11 & 0 &  00 & 0 & 0 & 0 \\
QE226 & optimal & 0.19 & 0 & 4.36e-107 & 8 & 0 & 0 \\
QETAMACR & optimal & 1.24 & 0 & 3.85e-106 & 10 & 0 & 0 \\
QFFFFF80 & optimal & 0.99 & 0 & 3.05e-114 & 9 & 0 & 0 \\
QFORPLAN & optimal & 0.23 & 0 & 1.88e-113 & 8 & 0 & 0 \\
QGFRDXPN & optimal & 0.45 & 0 & 3.34e-104 & 9 & 0 & 0 \\
QGROW15 & optimal & 0.37 & 0 &  00 & 0 & 0 & 0 \\
QGROW22 & optimal & 0.90 & 0 &  00 & 0 & 0 & 0 \\
QGROW7 & optimal & 0.22 & 0 &  00 & 0 & 0 & 0 \\
QISRAEL & optimal & 0.11 & 0 &  00 & 0 & 0 & 0 \\
QPCBLEND & optimal & 0.03 & 0 &  00 & 0 & 0 & 0 \\
QPCBOEI1 & optimal & 0.58 & 0 & 5.70e-114 & 10 & 0 & 0 \\
QPCBOEI2 & optimal & 0.10 & 0 &  00 & 0 & 0 & 0 \\
QPCSTAIR & optimal & 0.96 & 0 & 1.04e-109 & 11 & 0 & 0 \\
QPILOTNO & inconsistent & 40.49 & 0 & NaN & 50 & 0 & 0 \\
QPTEST & optimal & 0.00 & 0 &  00 & 0 & 0 & 0 \\
QRECIPE & inconsistent & 0.04 & 0 & NaN & 50 & 0 & 0 \\
QSC205 & optimal & 0.30 & 0 & 1.27e-107 & 11 & 0 & 0 \\
QSCAGR25 & optimal & 0.32 & 0 & 1.06e-104 & 7 & 0 & 0 \\
QSCAGR7 & optimal & 0.02 & 0 & 3.71e-107 & 8 & 0 & 0 \\
QSCFXM1 & optimal & 0.20 & 0 & 1.79e-112 & 7 & 0 & 0 \\
QSCFXM2 & optimal & 0.73 & 0 & 5.01e-112 & 9 & 0 & 0 \\
QSCFXM3 & optimal & 1.46 & 0 & 3.59e-101 & 8 & 0 & 0 \\
QSCORPIO & optimal & 0.16 & 0 & 2.31e-101 & 8 & 0 & 0 \\
QSCRS8 & optimal & 0.36 & 0 & 1.58e-108 & 13 & 0 & 0 \\
QSCSD1 & optimal & 0.12 & 0 & 7.72e-110 & 7 & 0 & 0 \\
QSCSD6 & optimal & 0.18 & 0 & 2.69e-111 & 7 & 0 & 0 \\
QSCSD8 & optimal & 1.16 & 0 &  00 & 0 & 0 & 0 \\
QSCTAP1 & optimal & 0.18 & 0 & 6.15e-106 & 7 & 0 & 0 \\
QSCTAP2 & optimal & 0.55 & 0 & 5.86e-111 & 6 & 0 & 0 \\
QSCTAP3 & optimal & 1.04 & 0 &  00 & 0 & 0 & 0 \\
QSEBA & optimal & 0.24 & 0 & 1.15e-103 & 9 & 0 & 0 \\
QSHARE1B & optimal & 0.07 & 0 &  00 & 0 & 0 & 0 \\
QSHARE2B & optimal & 0.04 & 0 &  00 & 0 & 0 & 0 \\
QSHELL & optimal & 0.82 & 0 & 3.16e-103 & 9 & 0 & 0 \\
QSHIP04L & optimal & 0.60 & 0 & 7.13e-110 & 9 & 0 & 0 \\
QSHIP04S & optimal & 0.43 & 0 & 5.88e-109 & 9 & 0 & 0 \\
QSHIP08L & optimal & 2.14 & 0 & 2.31e-110 & 7 & 0 & 0 \\
QSHIP08S & optimal & 1.21 & 0 & 2.31e-108 & 9 & 0 & 0 \\
QSHIP12L & optimal & 2.78 & 0 & 3.71e-113 & 8 & 0 & 0 \\
QSHIP12S & optimal & 3.00 & 0 & 2.82e-104 & 9 & 0 & 0 \\
QSIERRA & error & 0.95 & 0 & NaN & 6 & 0 & 0 \\
QSTAIR & optimal & 0.58 & 0 & 4.32e-101 & 8 & 0 & 0 \\
QSTANDAT & optimal & 0.06 & 0 & 3.62e-107 & 7 & 0 & 0 \\
S268 & inconsistent & 0.00 & 0 &  00 & 0 & 0 & 0 \\
STADAT1 & optimal & 9.82 & 0 &  00 & 0 & 0 & 0 \\
STADAT2 & optimal & 9.89 & 0 &  00 & 0 & 0 & 0 \\
STADAT3 & optimal & 40.23 & 0 &  00 & 0 & 0 & 0 \\
STCQP1 & optimal & 47.09 & 0 & 1.91e-107 & 11 & 0 & 0 \\
STCQP2 & optimal & 25.77 & 0 & 5.65e-102 & 9 & 0 & 0 \\
TAME & optimal & 0.00 & 0 &  00 & 0 & 0 & 0 \\
UBH1 & optimal & 136.35 & 0 &  00 & 0 & 0 & 0 \\
VALUES & optimal & 0.03 & 0 &  00 & 0 & 0 & 0 \\
YAO & optimal & 1.53 & 0 &  00 & 0 & 0 & 0 \\
ZECEVIC2 & optimal & 0.00 & 0 &  00 & 0 & 0 & 0 \\
\end{longtable}
\clearpage{}
}
{\small
\clearpage{}\begin{longtable}{l|l|l|l|l|l|l|l}
\caption{Detailed results for exact solve of large QP set with parameter set~s4 (Iter.=Iterations, Tol.=Tolerance, Ref.=Refinements, Back.=Backstepping, Res.=Resolves)}
\label{tab:resultsexactqorenofaclarge}\\
\toprule
\rowcolor[gray]{0.85}
QP Name & Status & Time & Iter. & Tol. & Ref. & Back. & Res. \\
\rowcolor[gray]{0.85}
&        & [s]  & [\#]     & [-]       & [\#]      & [\#]   & [\#]  \\
\midrule
\endfirsthead
\toprule
\rowcolor[gray]{0.85}
QP Name & Status & Time & Iter. & Tol. & Ref. & Back. & Res. \\
\rowcolor[gray]{0.85}
&        & [s]  & [\#]     & [-]       & [\#]      & [\#]   & [\#]  \\
\midrule
\endhead

\midrule
\multicolumn{8}{r}{continued on next page}\\
\bottomrule
\endfoot

\bottomrule
\endlastfoot
AUG2D & error & 55.14 & 0 & NaN & 1 & 0 & 0 \\
AUG2DC & optimal & 56.77 & 0 & 4.76e-109 & 7 & 0 & 0 \\
AUG2DCQP & optimal & 154.49 & 0 & 1.47e-101 & 8 & 0 & 0 \\
AUG2DQP & optimal & 320.54 & 0 & 3.10e-107 & 7 & 0 & 0 \\
AUG3D & error & 1.54 & 0 & NaN & 1 & 0 & 0 \\
AUG3DC & optimal & 6.38 & 0 & 7.59e-108 & 7 & 0 & 0 \\
AUG3DCQP & optimal & 4.73 & 0 & 1.52e-102 & 6 & 0 & 0 \\
AUG3DQP & optimal & 3.31 & 0 & 8.00e-103 & 6 & 0 & 0 \\
BOYD2 & abort & NaN & NaN & NaN & NaN & NaN & NaN \\
CONT-050 & optimal & 53.31 & 0 & 2.25e-108 & 8 & 0 & 0 \\
CONT-100 & optimal & 1076.87 & 0 & 2.86e-102 & 7 & 0 & 0 \\
CONT-101 & optimal & 1029.04 & 0 & 5.89e-101 & 7 & 0 & 0 \\
CONT-200 & timeout & 13506.53 & 0 & NaN & 4 & 0 & 0 \\
CONT-201 & timeout & 11449.02 & 0 & NaN & 3 & 0 & 0 \\
CONT-300 & abort & NaN & NaN & NaN & NaN & NaN & NaN \\
CVXQP1\_L & optimal & 6853.92 & 0 & 1.90e-101 & 9 & 0 & 0 \\
CVXQP1\_M & optimal & 5.74 & 0 & 4.34e-107 & 9 & 0 & 0 \\
CVXQP1\_S & optimal & 0.05 & 0 & 8.55e-102 & 8 & 0 & 0 \\
CVXQP2\_L & optimal & 550.82 & 0 & 9.00e-109 & 9 & 0 & 0 \\
CVXQP2\_M & optimal & 0.98 & 0 & 1.12e-108 & 9 & 0 & 0 \\
CVXQP2\_S & optimal & 0.02 & 0 & 2.25e-110 & 7 & 0 & 0 \\
CVXQP3\_L & error & 7293.07 & 0 & NaN & 6 & 1 & 0 \\
CVXQP3\_M & optimal & 18.06 & 0 & 3.73e-111 & 9 & 0 & 0 \\
CVXQP3\_S & reached & 0.45 & 0 & NaN & 50 & 0 & 0 \\
DPKLO1 & optimal & 0.09 & 0 & 8.39e-108 & 9 & 0 & 0 \\
DTOC3 & optimal & 112.18 & 0 & 1.48e-114 & 7 & 0 & 0 \\
DUAL1 & optimal & 0.07 & 0 & 3.21e-103 & 10 & 0 & 0 \\
DUAL2 & optimal & 0.06 & 0 & 1.54e-101 & 6 & 0 & 0 \\
DUAL3 & optimal & 0.06 & 0 & 2.48e-103 & 6 & 0 & 0 \\
DUAL4 & optimal & 0.04 & 0 & 1.37e-104 & 6 & 0 & 0 \\
DUALC1 & optimal & 0.05 & 0 & 8.05e-106 & 6 & 0 & 0 \\
DUALC2 & optimal & 0.07 & 0 & 1.44e-106 & 6 & 0 & 0 \\
DUALC5 & optimal & 0.06 & 0 & 6.55e-107 & 6 & 0 & 0 \\
DUALC8 & optimal & 0.22 & 0 & 7.30e-105 & 6 & 0 & 0 \\
EXDATA & error & 560.71 & 0 & NaN & 5 & 0 & 0 \\
GENHS28 & optimal & 0.00 & 0 & 1.25e-110 & 6 & 0 & 0 \\
GOULDQP2 & inconsistent & 0.22 & 0 & 1.50e-110 & 7 & 0 & 0 \\
GOULDQP3 & optimal & 0.36 & 0 & 4.48e-108 & 11 & 0 & 0 \\
HS118 & inconsistent & 0.02 & 0 & NaN & 50 & 0 & 0 \\
HS21 & optimal & 0.00 & 0 &  00 & 0 & 0 & 0 \\
HS268 & inconsistent & 0.00 & 0 & 6.35e-112 & 9 & 0 & 0 \\
HS35 & optimal & 0.00 & 0 & 9.00e-106 & 6 & 0 & 0 \\
HS35MOD & optimal & 0.00 & 0 & 4.03e-108 & 6 & 0 & 0 \\
HS51 & optimal & 0.00 & 0 & 3.14e-108 & 6 & 0 & 0 \\
HS52 & optimal & 0.00 & 0 & 4.97e-106 & 6 & 0 & 0 \\
HS53 & optimal & 0.00 & 0 & 2.34e-108 & 6 & 0 & 0 \\
HS76 & optimal & 0.00 & 0 & 3.83e-108 & 6 & 0 & 0 \\
HUESTIS & inconsistent & 37.99 & 0 & 6.05e-104 & 7 & 0 & 0 \\
HUES-MOD & inconsistent & 18.39 & 0 & 4.21e-108 & 7 & 0 & 0 \\
KSIP & optimal & 2.11 & 0 & 3.89e-105 & 6 & 0 & 0 \\
LASER & optimal & 2.94 & 0 & 7.30e-105 & 12 & 0 & 0 \\
LISWET1 & optimal & 768.27 & 0 & 1.01e-107 & 15 & 0 & 0 \\
LISWET10 & optimal & 638.05 & 0 & 1.60e-105 & 15 & 0 & 0 \\
LISWET11 & optimal & 586.91 & 0 & 4.47e-103 & 13 & 0 & 0 \\
LISWET12 & optimal & 574.81 & 0 & 6.47e-106 & 12 & 0 & 0 \\
LISWET2 & optimal & 699.25 & 0 & 4.17e-107 & 17 & 0 & 0 \\
LISWET3 & optimal & 299.09 & 0 & 1.20e-101 & 14 & 0 & 0 \\
LISWET4 & optimal & 309.00 & 0 & 2.60e-107 & 15 & 0 & 0 \\
LISWET5 & optimal & 312.78 & 0 & 4.27e-106 & 15 & 0 & 0 \\
LISWET6 & optimal & 316.60 & 0 & 3.68e-102 & 14 & 0 & 0 \\
LISWET7 & optimal & 1189.62 & 0 & 3.11e-108 & 37 & 0 & 0 \\
LISWET8 & inconsistent & 659.38 & 0 & 1.63e-104 & 15 & 0 & 0 \\
LISWET9 & optimal & 671.45 & 0 & 2.04e-104 & 10 & 0 & 0 \\
LOTSCHD & optimal & 0.00 & 0 & 2.70e-107 & 6 & 0 & 0 \\
MOSARQP1 & optimal & 2.14 & 0 & 2.01e-115 & 10 & 0 & 0 \\
MOSARQP2 & optimal & 0.83 & 0 & 1.11e-107 & 9 & 0 & 0 \\
POWELL20 & optimal & 127.81 & 0 & 2.82e-110 & 9 & 0 & 0 \\
PRIMAL1 & optimal & 0.14 & 0 & 1.35e-112 & 7 & 0 & 0 \\
PRIMAL2 & optimal & 0.27 & 0 & 2.38e-103 & 7 & 0 & 0 \\
PRIMAL3 & optimal & 0.39 & 0 & 1.38e-102 & 6 & 0 & 0 \\
PRIMAL4 & optimal & 0.72 & 0 & 1.06e-112 & 7 & 0 & 0 \\
PRIMALC1 & optimal & 0.04 & 0 & 1.56e-104 & 6 & 0 & 0 \\
PRIMALC2 & optimal & 0.01 & 0 & 8.59e-107 & 6 & 0 & 0 \\
PRIMALC5 & optimal & 0.01 & 0 & 1.63e-103 & 6 & 0 & 0 \\
PRIMALC8 & inconsistent & 0.15 & 0 & NaN & 50 & 0 & 0 \\
Q25FV47 & optimal & 9.15 & 0 & 3.75e-114 & 9 & 0 & 0 \\
QADLITTL & optimal & 0.03 & 0 & 2.11e-111 & 7 & 0 & 0 \\
QAFIRO & error & 0.01 & 0 & NaN & 9 & 0 & 0 \\
QBANDM & optimal & 0.24 & 0 & 1.17e-103 & 6 & 0 & 0 \\
QBEACONF & inconsistent & 0.21 & 0 & NaN & 50 & 0 & 0 \\
QBORE3D & inconsistent & 0.15 & 0 & NaN & 50 & 0 & 0 \\
QBRANDY & optimal & 0.18 & 0 & 4.04e-106 & 9 & 0 & 0 \\
QCAPRI & optimal & 0.27 & 0 & 8.03e-116 & 10 & 0 & 0 \\
QE226 & optimal & 0.21 & 0 & 4.36e-107 & 8 & 0 & 0 \\
QETAMACR & optimal & 1.24 & 0 & 3.85e-106 & 10 & 0 & 0 \\
QFFFFF80 & optimal & 0.97 & 0 & 3.05e-114 & 9 & 0 & 0 \\
QFORPLAN & optimal & 0.23 & 0 & 1.88e-113 & 8 & 0 & 0 \\
QGFRDXPN & optimal & 0.41 & 0 & 3.34e-104 & 9 & 0 & 0 \\
QGROW15 & optimal & 0.32 & 0 & 2.68e-107 & 9 & 0 & 0 \\
QGROW22 & optimal & 0.58 & 0 & 1.77e-107 & 9 & 0 & 0 \\
QGROW7 & optimal & 0.10 & 0 & 4.66e-109 & 9 & 0 & 0 \\
QISRAEL & optimal & 0.19 & 0 & 9.15e-104 & 9 & 0 & 0 \\
QPCBLEND & optimal & 0.03 & 0 & 4.29e-114 & 9 & 0 & 0 \\
QPCBOEI1 & optimal & 0.50 & 0 & 5.70e-114 & 10 & 0 & 0 \\
QPCBOEI2 & optimal & 0.12 & 0 & 6.52e-105 & 9 & 0 & 0 \\
QPCSTAIR & optimal & 0.95 & 0 & 1.04e-109 & 11 & 0 & 0 \\
QPILOTNO & inconsistent & 40.22 & 0 & NaN & 50 & 0 & 0 \\
QPTEST & optimal & 0.00 & 0 & 2.66e-110 & 6 & 0 & 0 \\
QRECIPE & inconsistent & 0.02 & 0 & NaN & 50 & 0 & 0 \\
QSC205 & optimal & 0.27 & 0 & 1.27e-107 & 11 & 0 & 0 \\
QSCAGR25 & optimal & 0.32 & 0 & 1.06e-104 & 7 & 0 & 0 \\
QSCAGR7 & optimal & 0.02 & 0 & 3.71e-107 & 8 & 0 & 0 \\
QSCFXM1 & optimal & 0.16 & 0 & 1.79e-112 & 7 & 0 & 0 \\
QSCFXM2 & optimal & 0.68 & 0 & 5.01e-112 & 9 & 0 & 0 \\
QSCFXM3 & optimal & 1.50 & 0 & 3.59e-101 & 8 & 0 & 0 \\
QSCORPIO & optimal & 0.16 & 0 & 2.31e-101 & 8 & 0 & 0 \\
QSCRS8 & optimal & 0.36 & 0 & 1.58e-108 & 13 & 0 & 0 \\
QSCSD1 & optimal & 0.09 & 0 & 7.72e-110 & 7 & 0 & 0 \\
QSCSD6 & optimal & 0.23 & 0 & 2.69e-111 & 7 & 0 & 0 \\
QSCSD8 & optimal & 8.62 & 0 & 2.23e-106 & 10 & 1 & 0 \\
QSCTAP1 & optimal & 0.19 & 0 & 6.15e-106 & 7 & 0 & 0 \\
QSCTAP2 & optimal & 0.55 & 0 & 5.86e-111 & 6 & 0 & 0 \\
QSCTAP3 & optimal & 0.92 & 0 & 7.96e-110 & 6 & 0 & 0 \\
QSEBA & optimal & 0.30 & 0 & 1.15e-103 & 9 & 0 & 0 \\
QSHARE1B & optimal & 0.04 & 0 & 5.84e-112 & 7 & 0 & 0 \\
QSHARE2B & optimal & 0.05 & 0 & 1.65e-105 & 9 & 0 & 0 \\
QSHELL & optimal & 0.80 & 0 & 3.16e-103 & 9 & 0 & 0 \\
QSHIP04L & optimal & 0.58 & 0 & 7.13e-110 & 9 & 0 & 0 \\
QSHIP04S & optimal & 0.43 & 0 & 5.88e-109 & 9 & 0 & 0 \\
QSHIP08L & optimal & 2.06 & 0 & 2.31e-110 & 7 & 0 & 0 \\
QSHIP08S & optimal & 1.18 & 0 & 2.31e-108 & 9 & 0 & 0 \\
QSHIP12L & optimal & 2.67 & 0 & 3.71e-113 & 8 & 0 & 0 \\
QSHIP12S & optimal & 2.92 & 0 & 2.82e-104 & 9 & 0 & 0 \\
QSIERRA & error & 0.95 & 0 & NaN & 6 & 0 & 0 \\
QSTAIR & optimal & 0.66 & 0 & 4.32e-101 & 8 & 0 & 0 \\
QSTANDAT & optimal & 0.11 & 0 & 3.62e-107 & 7 & 0 & 0 \\
S268 & inconsistent & 0.00 & 0 & 6.35e-112 & 9 & 0 & 0 \\
STADAT1 & optimal & 15.69 & 0 & 2.28e-107 & 8 & 0 & 0 \\
STADAT2 & optimal & 16.69 & 0 & 2.33e-110 & 11 & 0 & 0 \\
STADAT3 & optimal & 116.94 & 0 & 8.80e-104 & 12 & 1 & 0 \\
STCQP1 & optimal & 47.31 & 0 & 1.91e-107 & 11 & 0 & 0 \\
STCQP2 & optimal & 25.51 & 0 & 5.65e-102 & 9 & 0 & 0 \\
TAME & optimal & 0.00 & 0 & 3.12e-112 & 6 & 0 & 0 \\
UBH1 & optimal & 203.99 & 0 & 3.28e-111 & 7 & 0 & 0 \\
VALUES & inconsistent & 0.08 & 0 & NaN & 50 & 0 & 0 \\
YAO & optimal & 23.19 & 0 & 1.58e-106 & 10 & 0 & 0 \\
ZECEVIC2 & optimal & 0.00 & 0 & 2.66e-110 & 6 & 0 & 0 \\
\end{longtable}
\clearpage{}
}
\begin{landscape}
{\small
\clearpage{}\begin{longtable}[c]{l|lll|lll|lll|lll}
\caption{Detailed results for inexact and fast solves of medium QP set with parameter set s5 and the three standard \texttt{qpOASES} option sets (Iter.=Iterations, Tol.=Tolerance)}
\label{tab:resultsinexactlarge}\\
\toprule
\rowcolor[gray]{0.85}
        & \multicolumn{3}{c|}{Refinement} & \multicolumn{9}{|c}{\texttt{qpOASES} with standard settings:}                                      \\
\rowcolor[gray]{0.85}
        & \multicolumn{3}{c|}{Tol. 1e-12} & \multicolumn{3}{|c}{MPC}   & \multicolumn{3}{c}{Default} & \multicolumn{3}{c}{Reliable} \\
\midrule
\rowcolor[gray]{0.85}
QP Name & Time & Iter. (Ref.)  & Tol.       & Time & Iter. & Tol.         & Time & Iter. & Tol.          & Time & Iter. & Tol.           \\
\rowcolor[gray]{0.85}
        & [s]  & [\#] ([\#]) & [-]        & [s]  & [\#] & [-]          & [s]  & [\#] & [-]           & [s]  & [\#] & [-]            \\
\midrule
\endfirsthead
\toprule
\rowcolor[gray]{0.85}
        & \multicolumn{3}{c|}{Refinement} & \multicolumn{9}{|c}{\texttt{qpOASES} with standard settings:}                                      \\
\rowcolor[gray]{0.85}
        & \multicolumn{3}{c|}{Tol. 1e-12} & \multicolumn{3}{|c}{MPC}   & \multicolumn{3}{c}{Default} & \multicolumn{3}{c}{Reliable} \\
\midrule
\rowcolor[gray]{0.85}
QP Name & Time & Iter. (Ref.)  & Tol.       & Time & Iter. & Tol.         & Time & Iter. & Tol.          & Time & Iter. & Tol.           \\
\rowcolor[gray]{0.85}
        & [s]  & [\#] ([\#]) & [-]        & [s]  & [\#] & [-]          & [s]  & [\#] & [-]           & [s]  & [\#] & [-]            \\
\midrule
\endhead

\midrule
\multicolumn{13}{r}{continued on next page}\\
\bottomrule
\endfoot

\bottomrule
\endlastfoot
CVXQP1\_M & 9.53 & 412 (1) & 8.50e-21 & 6.94 & 404 & 2.17e-10 & 23.92 & 1705 & 2.21e-11 & 169.24 & 1705 & 2.34e-11 \\
CVXQP1\_S & 0.02 & 36 (0) & 2.73e-13 & 0.02 & 36 & 4.05e-12 & 0.02 & 139 & 4.39e-13 & 0.06 & 139 & 3.88e-13 \\
CVXQP2\_M & 10.50 & 651 (0) & 1.05e-12 & 9.48 & 663 & 5.65e-11 & 6.84 & 712 & 1.50e-12 & 44.61 & 713 & 8.41e-13 \\
CVXQP2\_S & 0.02 & 62 (0) & 7.47e-14 & 0.02 & 60 & 1.29e-11 & 0.01 & 68 & 9.94e-14 & 0.01 & 68 & 5.82e-14 \\
CVXQP3\_M & 4.52 & 231 (2) & 8.10e-22 & 3.86 & 229 & 5.40e-09 & 79.12 & 3867 & 3.95e-10 & 307.04 & 3869 & 3.20e-10 \\
CVXQP3\_S & 0.02 & 22 (0) & 7.18e-12 & 0.02 & 24 & 5.30e-11 & 0.07 & 189 & 4.73e-13 & 0.05 & 181 & 2.50e-13 \\
DPKLO1 & 0.02 & 0 (0) & 2.67e-14 & 0.02 & 0 & 2.67e-14 & 0.14 & 321 & 5.75e-15 & 0.31 & 321 & 7.32e-15 \\
DUAL1 & 0.03 & 26 (1) & 5.14e-17 & 0.02 & 28 & 1.09e-11 & 0.04 & 75 & 8.78e-16 & 0.02 & 75 & 7.06e-16 \\
DUAL2 & 0.04 & 4 (1) & 6.32e-17 & 0.03 & 4 & 3.05e-11 & 0.05 & 92 & 6.06e-16 & 0.05 & 92 & 7.44e-16 \\
DUAL3 & 0.05 & 14 (1) & 2.53e-17 & 0.03 & 14 & 1.68e-12 & 0.06 & 97 & 1.03e-15 & 0.06 & 97 & 9.27e-16 \\
DUAL4 & 0.01 & 13 (0) & 1.01e-15 & 0.01 & 13 & 1.51e-11 & 0.03 & 62 & 1.04e-15 & 0.02 & 62 & 8.03e-16 \\
DUALC1 & 0.08 & 31 (0) & 4.40e-12 & 0.01 & 31 & 2.18e-10 & 0.01 & 4 & 6.20e-13 & 0.00 & 4 & 6.25e-13 \\
DUALC2 & 0.14 & 28 (1) & 1.91e-21 & 0.01 & 30 & 2.17e-09 & 0.00 & 5 & 2.33e-13 & 0.00 & 5 & 2.73e-13 \\
DUALC5 & 0.10 & 3 (0) & 3.32e-13 & 0.01 & 3 & 1.30e-09 & 0.01 & 5 & 2.68e-13 & 0.00 & 5 & 8.70e-14 \\
DUALC8 & 0.60 & 13 (1) & 6.10e-19 & 0.02 & 11 & 8.95e-10 & 0.01 & 6 & 5.15e-11 & 0.01 & 0 &  NaN \\
GENHS28 & 0.00 & 0 (0) & 3.47e-16 & 0.00 & 0 & 3.47e-16 & 0.00 & 14 & 6.80e-16 & 0.00 & 14 & 5.13e-16 \\
GOULDQP2 & 6.53 & 585 (2) & 9.00e-21 & 3.77 & 573 & 7.77e-14 & 15.44 & 2409 & 1.13e-09 & 247.95 & 2409 & 1.13e-09 \\
GOULDQP3 & 1.68 & 176 (0) & 7.52e-15 & 1.54 & 176 & 5.94e-13 & 2.98 & 740 & 8.40e-15 & 17.52 & 740 & 8.16e-15 \\
HS118 & 0.00 & 24 (0) & 4.71e-15 & 0.00 & 24 & 6.76e-13 & 0.00 & 27 & 6.92e-15 & 0.00 & 27 & 5.56e-15 \\
HS21 & 0.00 & 3 (0) & 8.80e-16 & 0.00 & 3 & 1.28e-17 & 0.00 & 1 & 5.55e-17 & 0.00 & 1 & 5.55e-17 \\
HS268 & 0.00 & 0 (1) & 1.69e-16 & 0.00 & 0 & 2.62e-12 & 0.00 & 11 & 8.60e-07 & 0.00 & 12 & 7.87e-13 \\
HS35 & 0.00 & 1 (0) & 3.61e-16 & 0.00 & 1 & 2.22e-16 & 0.00 & 4 & 7.77e-16 & 0.00 & 4 & 3.06e-16 \\
HS35MOD & 0.00 & 0 (0) & 1.01e-15 & 0.00 & 0 & 9.98e-16 & 0.00 & 4 & 1.66e-16 & 0.00 & 4 & 1.66e-16 \\
HS51 & 0.00 & 0 (0) & 2.54e-16 & 0.00 & 0 & 2.54e-16 & 0.00 & 5 & 8.16e-15 & 0.00 & 5 & 8.16e-15 \\
HS52 & 0.00 & 0 (0) & 1.22e-15 & 0.00 & 0 & 1.22e-15 & 0.00 & 10 & 2.29e-15 & 0.00 & 10 & 1.13e-15 \\
HS53 & 0.00 & 0 (0) & 3.89e-16 & 0.00 & 0 & 3.89e-16 & 0.00 & 9 & 1.00e-15 & 0.00 & 9 & 8.30e-16 \\
HS76 & 0.00 & 4 (0) & 3.16e-17 & 0.00 & 4 & 2.78e-16 & 0.00 & 4 & 6.36e-16 & 0.00 & 4 & 6.36e-16 \\
KSIP & 24.26 & 1080 (1) & 1.63e-18 & 0.15 & 1088 & 3.61e-15 & 0.21 & 1019 & 1.89e-16 & 0.20 & 1019 & 3.68e-17 \\
LOTSCHD & 0.00 & 5 (0) & 8.82e-15 & 0.00 & 5 & 5.30e-13 & 0.00 & 18 & 1.82e-14 & 0.00 & 18 & 2.27e-14 \\
MOSARQP2 & 18.25 & 332 (0) & 1.18e-15 & 2.54 & 332 & 6.33e-13 & 8.99 & 1012 & 1.83e-15 & 317.07 & 1012 & 1.32e-15 \\
PRIMAL1 & 0.38 & 77 (0) & 1.92e-16 & 0.17 & 75 & 1.76e-11 & 0.50 & 399 & 2.29e-16 & 6.14 & 399 & 1.05e-15 \\
PRIMAL2 & 3.49 & 94 (0) & 8.16e-16 & 0.52 & 96 & 4.02e-11 & 3.04 & 742 & 1.45e-15 & 85.84 & 742 & 1.68e-15 \\
PRIMAL3 & 2.95 & 117 (0) & 1.21e-15 & 0.69 & 101 & 4.77e-11 & 4.86 & 841 & 1.48e-15 & 134.65 & 841 & 1.83e-16 \\
PRIMALC1 & 0.17 & 234 (1) & 2.69e-22 & 0.10 & 222 & 2.47e-09 & 0.02 & 27 & 6.50e-13 & 0.01 & 27 & 1.01e-12 \\
PRIMALC2 & 0.16 & 237 (1) & 1.01e-22 & 0.11 & 237 & 6.00e-13 & 0.00 & 10 & 8.05e-16 & 0.00 & 10 & 1.66e-13 \\
PRIMALC5 & 0.27 & 288 (1) & 7.11e-23 & 0.17 & 292 & 3.15e-10 & 0.03 & 23 & 2.50e-14 & 0.02 & 23 & 2.58e-14 \\
PRIMALC8 & 2.76 & 515 (1) & 3.87e-17 & 0.70 & 519 & 2.34e-09 & 0.07 & 25 & 1.91e-14 & 0.05 & 26 & 4.08e-12 \\
QADLITTL & 0.05 & 234 (1) & 1.92e-17 & 0.03 & 189 & 5.93e-09 & 0.02 & 132 & 3.31e-13 & 0.02 & 124 & 1.39e-10 \\
QAFIRO & 0.00 & 30 (0) & 5.49e-15 & 0.00 & 29 & 8.37e-11 & 0.00 & 16 & 3.48e-15 & 0.00 & 16 & 3.48e-15 \\
QBANDM & 2.85 & 1164 (0) & 1.54e-13 & 1.59 & 870 & 1.61e-08 & 5.43 & 1512 & 2.23e-13 & 7.69 & 1512 & 9.62e-14 \\
QBEACONF & 0.43 & 304 (2) & 1.18e-18 & 0.24 & 302 & 2.44e-08 & 0.09 & 133 & 1.30e-11 & 0.09 & 133 & 1.31e-11 \\
QBORE3D & 0.78 & 522 (1) & 5.27e-13 & 0.25 & 155 &  NaN & 0.30 & 221 & 4.61e-13 & 0.31 & 221 & 2.76e-11 \\
QBRANDY & 0.56 & 444 (0) & 4.69e-12 & 0.28 & 414 & 2.53e-08 & 0.85 & 854 & 2.43e-13 & 1.16 & 854 & 3.13e-13 \\
QCAPRI & 2.69 & 1050 (1) & 3.62e-19 & 1.06 & 1008 & 9.68e-08 & 0.64 & 457 & 4.83e-10 & 0.91 & 457 & 4.34e-10 \\
QE226 & 3.24 & 1393 (1) & 1.07e-13 & 0.95 & 1431 & 2.59e-08 & 1.17 & 825 & 2.36e-14 & 2.07 & 813 & 3.78e-14 \\
QETAMACR & 4.12 & 558 (2) & 6.05e-16 & 2.15 & 493 & 2.14e-09 & 7.86 & 1354 & 2.30e-07 & 23.13 & 1354 & 2.30e-07 \\
QFFFFF80 & 19.05 & 1008 (1) & 2.26e-13 & 7.76 & 1005 & 9.90e-08 & 20.85 & 2293 & 1.59e-11 & 73.22 & 1722 &  NaN \\
QFORPLAN & 2.21 & 1180 (1) & 7.31e-22 & 1.53 & 1930 & 2.39e+06 & 1.28 & 796 & 8.03e-10 & 1.81 & 804 & 7.13e-10 \\
QGROW15 & 4.07 & 629 (1) & 8.55e-18 & 2.66 & 531 & 1.04e-07 & 3.29 & 600 & 4.28e-09 & 4.17 & 589 & 6.40e-09 \\
QGROW22 & 15.30 & 934 (2) & 1.52e-20 & 7.62 & 621 & 1.14e-07 & 10.71 & 888 & 3.69e-07 & 14.64 & 881 & 4.97e-08 \\
QGROW7 & 0.45 & 296 (1) & 1.56e-19 & 0.28 & 266 & 5.07e-08 & 0.42 & 340 & 1.88e-09 & 0.45 & 298 & 3.22e-09 \\
QISRAEL & 0.41 & 290 (0) & 3.11e-12 & 0.11 & 360 & 3.99e-08 & 0.08 & 258 & 4.81e-12 & 0.12 & 258 & 5.03e-12 \\
QPCBLEND & 0.03 & 66 (1) & 2.25e-16 & 0.01 & 64 & 8.38e-14 & 0.02 & 176 & 7.56e-16 & 0.02 & 176 & 9.67e-16 \\
QPCBOEI1 & 3.18 & 435 (2) & 2.12e-13 & 1.08 & 441 & 2.15e-09 & 1.32 & 652 & 2.74e-11 & 7.68 & 652 & 3.20e-11 \\
QPCBOEI2 & 0.30 & 175 (0) & 6.07e-11 & 0.07 & 177 & 7.37e-10 & 0.07 & 224 & 4.21e-10 & 0.16 & 224 & 3.00e-10 \\
QPCSTAIR & 1.32 & 257 (0) & 1.58e-11 & 0.54 & 239 & 9.26e-11 & 2.29 & 908 & 8.77e-12 & 3.67 & 908 & 8.06e-12 \\
QPTEST & 0.00 & 1 (0) & 7.40e-16 & 0.00 & 1 & 1.22e-15 & 0.00 & 2 & 3.89e-16 & 0.00 & 2 & 3.89e-16 \\
QRECIPE & 0.07 & 48 (0) & 5.90e-15 & 0.02 & 42 & 1.04e-10 & 0.02 & 88 & 1.64e-14 & 0.02 & 88 & 1.64e-14 \\
QSC205 & 0.29 & 93 (1) & 4.22e-18 & 0.04 & 43 & 7.63e-08 & 0.08 & 215 & 3.84e-16 & 0.08 & 215 & 4.35e-16 \\
QSCAGR25 & 7.38 & 1077 (2) & 1.77e-18 & 36.34 & 0 &  NaN & 6.12 & 1327 & 1.09e-11 & 8.88 & 1299 & 1.43e-11 \\
QSCAGR7 & 0.19 & 329 (0) & 3.88e-12 & 0.06 & 349 & 3.18e-10 & 0.22 & 418 & 4.58e-12 & 0.22 & 418 & 5.03e-12 \\
QSCFXM1 & 4.18 & 799 (2) & 1.25e-15 & 1.73 & 988 & 1.32e-07 & 1.69 & 580 & 9.14e-12 & 2.69 & 612 & 1.48e-11 \\
QSCFXM2 & 42.74 & 1890 (2) & 3.82e-18 & 16.53 & 2189 & 1.62e-07 & 17.45 & 1454 & 6.84e-11 & 26.22 & 1370 & 3.92e-11 \\
QSCORPIO & 0.87 & 233 (0) & 5.22e-14 & 0.12 & 0 &  NaN & 0.87 & 470 & 3.20e-11 & 1.19 & 469 & 1.06e-11 \\
QSCSD1 & 8.81 & 2054 (1) & 3.44e-22 & 6.51 & 1075 & 8.77e-11 & 0.64 & 204 & 4.77e-13 & 0.57 & 153 & 1.38e-10 \\
QSCTAP1 & 6.77 & 1374 (1) & 7.45e-15 & 30.15 & 0 &  NaN & 1.00 & 500 & 9.20e-12 & 1.44 & 503 & 1.54e-08 \\
QSHARE1B & 0.59 & 782 (2) & 1.25e-16 & 0.18 & 468 & 2.49e-08 & 0.34 & 517 & 1.21e-09 & 0.51 & 497 & 3.70e-11 \\
QSHARE2B & 0.12 & 359 (1) & 2.00e-13 & 0.02 & 355 & 5.33e-10 & 0.06 & 207 & 1.35e-12 & 0.05 & 196 & 1.25e-12 \\
QSTAIR & 3.18 & 792 (0) & 6.37e-12 & 1.28 & 784 & 1.92e-08 & 1.68 & 740 & 4.39e-12 & 2.67 & 740 & 4.41e-12 \\
S268 & 0.00 & 0 (1) & 1.69e-16 & 0.00 & 0 & 2.62e-12 & 0.00 & 11 & 8.60e-07 & 0.00 & 12 & 7.87e-13 \\
TAME & 0.00 & 0 (0) & 1.11e-16 & 0.00 & 0 & 1.11e-16 & 0.00 & 2 & 1.95e-16 & 0.00 & 2 & 1.95e-16 \\
VALUES & 0.10 & 142 (0) & 1.46e-12 & 0.04 & 0 &  NaN & 0.05 & 142 & 4.29e-16 & 0.04 & 142 & 2.45e-15 \\
ZECEVIC2 & 0.00 & 3 (0) & 1.25e-16 & 0.00 & 3 & 1.15e-12 & 0.00 & 2 & 2.22e-16 & 0.00 & 2 & 2.22e-16 \\
\end{longtable}
\clearpage{}
}
\end{landscape}

\end{document}